\documentclass[11pt]{article}

\usepackage[margin=1.05in]{geometry}
\usepackage[T1]{fontenc}
\usepackage[utf8]{inputenc}
\usepackage{lmodern}
\usepackage{amsmath,amssymb,amsthm,mathtools}
\usepackage{aliascnt}
\usepackage{mathrsfs}
\usepackage{microtype}
\usepackage[colorlinks=true,linkcolor=blue,citecolor=blue,urlcolor=blue]{hyperref}
\usepackage[nameinlink,noabbrev]{cleveref}

\numberwithin{equation}{section}

\theoremstyle{plain}
\newtheorem{theorem}{Theorem}[section]

\newaliascnt{proposition}{theorem}
\newtheorem{proposition}[proposition]{Proposition}
\aliascntresetthe{proposition}

\newaliascnt{lemma}{theorem}
\newtheorem{lemma}[lemma]{Lemma}
\aliascntresetthe{lemma}

\newaliascnt{corollary}{theorem}
\newtheorem{corollary}[corollary]{Corollary}
\aliascntresetthe{corollary}

\theoremstyle{definition}
\newtheorem{example}[theorem]{Example}

\theoremstyle{remark}
\newaliascnt{remark}{theorem}
\newtheorem{remark}[remark]{Remark}
\aliascntresetthe{remark}

\Crefname{theorem}{Theorem}{Theorems}
\Crefname{theorem}{Theorem}{Theorems}

\Crefname{proposition}{Proposition}{Propositions}
\Crefname{proposition}{Proposition}{Propositions}

\Crefname{lemma}{Lemma}{Lemmas}
\Crefname{lemma}{Lemma}{Lemmas}

\Crefname{corollary}{Corollary}{Corollaries}
\Crefname{corollary}{Corollary}{Corollaries}

\Crefname{remark}{Remark}{Remarks}
\Crefname{remark}{Remark}{Remarks}

\DeclareMathOperator{\Ad}{Ad}
\DeclareMathOperator{\Tr}{Tr}
\DeclareMathOperator{\diag}{diag}
\DeclareMathOperator{\Lie}{Lie}
\DeclareMathOperator{\Hom}{Hom}
\DeclareMathOperator{\pr}{pr}
\DeclareMathOperator{\Herm}{Herm}
\DeclareMathOperator{\vol}{vol}

\newcommand{\C}{\mathbb C}
\newcommand{\R}{\mathbb R}
\newcommand{\Z}{\mathbb Z}
\newcommand{\SU}{\mathrm{SU}}
\newcommand{\U}{\mathrm{U}}
\newcommand{\gLie}{\mathfrak g}
\newcommand{\kLie}{\mathfrak k}
\newcommand{\lLie}{\mathfrak l}
\newcommand{\pLie}{\mathfrak p}
\newcommand{\tLie}{\mathfrak t}
\newcommand{\nc}{\mathrm{nc}}
\newcommand{\sr}{\mathrm{SR}}
\newcommand{\Wt}{W}

\begin{document}
	
	\title{Sub-Laplacians on Compact Lie Groups: Heat Kernels, Distance, and Zeta Determinants}
	
	\author{Wolfram Bauer, Zhicheng Han, Zhipeng Yang\thanks{Corresponding author:yangzhipeng326@163.com}}
	\date{}
	\maketitle
	\AtEndDocument{%
		\par
		\bigskip
		\bigskip
		\noindent
		\textbf{Wolfram Bauer:}\\[0.2em]
		\textsc{Institut f\"ur Analysis, Leibniz Universit\"at Hannover,
			30167 Hannover, Germany}\\[0.3em]
		\textit{E-mail address}: \texttt{bauer@math.uni-hannover.de}\\[1.5em]
		\noindent
		\textbf{Zhicheng Han:}\\[0.2em]
		\textsc{Institut f\"ur Analysis, Leibniz Universit\"at Hannover,
			30167 Hannover, Germany}\\[0.3em]
		\textit{E-mail address}: \texttt{hanzc@math.uni-hannover.de}\\[1.5em]
		\noindent
		\textbf{Zhipeng Yang:}\\[0.2em]
		\textsc{Department of Mathematics, Yunnan Key Laboratory of Modern Analytical Mathematics and Applications, Yunnan Normal University, 650500 Kunming, China}\\
		\textsc{Department of Mathematics: Analysis, Logic and Discrete Mathematics, Ghent University, 9000 Ghent, Belgium}\\[0.3em]
		\textit{E-mail address}: \texttt{yangzhipeng326@163.com}%
	}

	\date{}

	\maketitle
	
	\begin{abstract}
		We study heat kernels, sub-Riemannian distances, and spectral zeta functions of sub-Laplacians determined by closed connected subgroups of compact Lie groups. Combining Hall's inversion formula with the affine lattice expansion of the compact group heat kernel, we derive a Cartan integral representation involving the group's exponential lattice. For two-step compact Lie pairs, the full algebraic small-time heat trace expansion is determined, up to an exponentially small remainder, by two explicit constants $C_{G,L}$ and $\beta_{G,L}$. This expansion determines all heat coefficients, the poles and residues of the reduced spectral zeta function, and its values at nonpositive integers. For the transvective symmetric subclass, we prove uniform vertical asymptotics for the Carnot-Carath\'eodory distance; the leading coefficient $\mathfrak F_{G,K}(Z)$ is the attained minimum of a finite-dimensional singular value problem. For simply connected two-step pairs, we obtain an exact decomposition of the zeta-regularized determinant into local, lattice, and spectral terms, with exponential truncation estimates. We specialize these results to block subgroups of $\SU(N)$, recovering the classical $\SU(2)$ and CR sphere spectra.
	\end{abstract}
	
	
	
	\section{Introduction and main results}
	
	The heat kernel is one of the principal objects through which geometry and analysis interact. For an elliptic Laplacian it carries three kinds of information at once. Its off-diagonal exponential decay recovers the Riemannian distance through Varadhan's formula \cite{Varadhan1967}; its diagonal expansion is formed from local geometric invariants \cite{Gilkey}; and the heat semigroup determines complex powers, while its trace yields the spectral zeta function and the zeta-regularized determinant \cite{Seeley,RaySinger}. All three arise from the heat semigroup, but they are usually accessed by different methods: large deviations for the distance, parametrices for local coefficients, and spectral decompositions or Mellin transforms for global invariants. The corresponding logarithmic distance formula remains valid for bracket-generating sums of squares and recovers the associated Carnot-Carath\'eodory distance.
	
	For sums of squares satisfying H\"ormander's bracket-generating condition, the same questions arise in an anisotropic geometry. H\"ormander's theorem gives hypoellipticity \cite{Hormander1967}, while the lifting and nilpotent approximation of Rothschild and Stein identify a graded nilpotent group as the local model \cite{RothschildStein1976}. Gaveau's analysis of two-step groups connected the heat kernel with a least action principle \cite{Gaveau}, and L\'eandre extended the logarithmic heat kernel formula to degenerate diffusions \cite{LeandreUpper1987,LeandreLower1987}. Precise small-time expansions away from and on the diagonal were obtained by Ben Arous \cite{BenArous1988,BenArous1989}. Subsequent work clarified the roles of the sub-Riemannian cut locus, abnormal minimizers, and nilpotentization \cite{BarilariBoscainNeel2012,ColinHillairetTrelat2021}.
	
	The use of spectral zeta functions to define determinants was inspired by their analogues in number theory. Minakshisundaram and Pleijel used the Mellin transform of the heat trace to meromorphically continue the zeta function of the Laplace-Beltrami operator \cite{MinakshisundaramPleijel1949}, and Seeley's construction of complex powers extended this construction to elliptic operators \cite{Seeley}. Ray and Singer then used zeta-regularized determinants of Hodge Laplacians to define analytic torsion \cite{RaySinger}. Quillen's determinant line construction and the work of Osgood, Phillips, and Sarnak on determinants as functionals of Riemannian metrics further established their role in global spectral geometry \cite{Quillen1985,OsgoodPhillipsSarnak1988}. For elliptic operators, the small-time heat expansion determines the poles of the spectral zeta function, whereas the zeta-regularized determinant depends on the full heat trace and is therefore genuinely global.
	
	The passage from Laplace-type operators to hypoelliptic operators preserves the Mellin transform argument but requires a heat calculus adapted to the filtered geometry. Explicit determinant formulas in nilpotent and sub-Riemannian geometry have been obtained in highly structured models. Furutani and de Gosson treated Laplace determinants on Heisenberg manifolds and related them to Poisson summation and the Kronecker limit formula \cite{FurutaniDeGosson2003}. Bauer and Furutani computed the zeta-regularized determinant of a sub-Laplacian on $\mathbb S^3$ and later treated product sub-Riemannian manifolds using Poisson summation on an $\mathbb S^1$-factor \cite{BauerFurutani2008,BauerFurutani2010}. Spectral zeta functions on compact two-step and pseudo $H$-type nilmanifolds were subsequently studied in \cite{BauerFurutaniIwasaki2012,BauerFurutaniIwasaki2015,Fischer2022}. The general heat calculus determines the local meromorphic structure, while the explicit evaluation of the finite global part remains model-dependent. The problem addressed here is to compute that finite part for simply connected two-step compact Lie pairs.
	
	Compact Lie groups provide complementary global structure. Their ordinary bi-invariant heat kernels admit Peter-Weyl expansions, descriptions through the global exponential map, and affine lattice formulas. The wrapping construction of Dooley and Wildberger and its heat kernel form display the exponential lattice \cite{DooleyWildberger,Maher}. Hall's Segal-Bargmann transform and inversion formula identify the holomorphic continuation of the compact group heat kernel and invert the compact heat flow \cite{Hall1994,Hall1997}. A horizontal sub-Laplacian retains this root and lattice structure, while its short-time geometry is governed by a Carnot tangent group. Heat kernels on nilpotent models and compact nilmanifolds provide the corresponding local formulas \cite{Cygan1979,AgrachevBoscainGauthierRossi2009, BauerFurutaniIwasaki2012,BauerFurutaniIwasaki2015,Fischer2022}, and representation-theoretic decompositions describe compact group sub-Laplacians and their sub-Riemannian length spectra \cite{Domokos2015,DomokosKrauelPignoShanbromVanValkenburgh2018}. Explicit integral representations and small-time asymptotics are also available on canonical model geometries such as $\SU(2)$, $\mathrm{SL}(2,\mathbb R)$, CR spheres, quaternionic Hopf fibrations, and the rank-five structure on $\mathbb S^7$ \cite{BaudoinBonnefont2009,Bonnefont2012,BaudoinWang2013,BaudoinWang2014, BauerLaaroussiTarama2024}. We combine the nilpotent tangent geometry, the actual exponential lattice, and the branching spectrum to obtain the local heat expansion and the global determinant formula.
	
	The purpose of this paper is to connect these local and global structures on compact Lie pairs. Their horizontal geometries have noncommutative Carnot tangent groups, while roots, exponential lattices, and branching laws remain available globally. The starting point is an inversion formula for the horizontal heat kernel. For two-step pairs, the logarithmic small-time limit links the same kernel to the sub-Riemannian distance. In vertical directions, the exact heat kernel expression obtained here enables a uniform comparison of endpoints generated by the same horizontal control in the compact group and its nilpotent tangent group and makes this link quantitative. The resulting analysis connects heat propagation, vertical distance, local invariants, and global spectral data.
	
	A further motivation comes from measurement-based quantum computation. In that setting, a resource state may implement only a prescribed family of infinitesimal logical generators, while directions outside that family appear only through commutators. A small noisy implementation has a quadratic local error, and after discretizing a horizontal curve, the accumulated leading error is a Riemann sum for sub-Riemannian energy. Thus the same sub-Riemannian distance also gives the leading geometric error coefficient. This motivates the corresponding sub-Riemannian control problem on compact Lie pairs. The measurement-based model is studied explicitly in \cite{HantzkoAdhikaryRaussendorf2025}, and the use of sub-Riemannian distance agrees with the broader geometric view of quantum computation in \cite{NielsenDowlingGuDoherty2006,RaussendorfBriegel2001}.
	
	Let $G$ be a compact connected Lie group, let $\theta$ be an involutive automorphism, and put $K=(G^\theta)_0$. Let $L\subset K$ be a closed connected subgroup. Taking $\theta=\mathrm{id}$ includes an arbitrary closed connected subgroup of $G$. Fix a $\theta$-invariant $\Ad(G)$-invariant inner product, let $\lLie=\Lie(L)$, and put $\pLie=\lLie^\perp$. Thus
	\[
	\gLie=\lLie\oplus\pLie,
	\qquad
	m=\dim\pLie,
	\qquad
	r=\dim\lLie.
	\]
	For an orthonormal basis $X_1,\ldots,X_m$ of $\pLie$, define the horizontal sub-Laplacian and its heat kernel by
	\[
	\Delta_{\pLie}=\sum_{i=1}^m\widetilde X_i^{\,2},
	\qquad
	p_t=e^{t\Delta_{\pLie}}\delta_e,
	\qquad t>0.
	\]
	The associated Carnot-Carath\'eodory distance is
	\begin{equation}\label{eq:intro-sr-distance}
		d_{\sr}(g_0,g_1):=\inf_{\gamma}
		\left\{\int_0^1|\gamma(t)^{-1}\dot\gamma(t)|\,\mathrm dt:
		\gamma(0)=g_0,\ \gamma(1)=g_1,\
		\gamma^{-1}\dot\gamma\in\pLie\ \text{a.e.}\right\}.
	\end{equation}
	If $\Delta_{\lLie}$ denotes the Laplacian in the $L$-directions, then
	\[
	\Delta_G=\Delta_{\lLie}+\Delta_{\pLie},
	\]
	where $\Delta_G$ is the bi-invariant Laplacian. As proved in \Cref{lem:strong-commutation}, the nonnegative self-adjoint closures of the two summands strongly commute, and hence
	\[
	e^{t\Delta_G}=e^{t\Delta_{\lLie}}e^{t\Delta_{\pLie}}.
	\]
	Thus the horizontal heat kernel can be recovered by inverting the $L$-heat flow applied to the ordinary heat kernel of the Lie group $G$. Hall's inversion formula converts this operator identity into an integral over $\lLie$. The holomorphic continuation of the ordinary heat kernel is then written as an affine lattice theta quotient.
	
	The first objective of this paper is to obtain an explicit integral expression for the heat kernel. Let $q_t$ be the heat kernel of $e^{t\Delta_G}$, and let $q_t^{\C}$ be its holomorphic continuation to the compatible matrix complexification fixed in \Cref{sec:preliminaries}, and fix a maximal torus $T\subset L$. Write $\rho_L$ for the half-sum of the positive roots of $L$ and $j_L^{\nc}$ for the noncompact Jacobian defined in \Cref{sec:preliminaries}. The inversion representation is as follows.
	
	\begin{theorem}
		If $\pLie$ generates $\gLie$ as a Lie algebra, that is, $\Lie(\pLie)=\gLie$, then, for $g\in G$ and $t>0$,
		\[
		p_t(g)=e^{-t|\rho_L|^2}
		\int_{\lLie}q_t^{\C}(ge^{iY})j_L^{\nc}(Y)^{1/2}
		\frac{e^{-|Y|^2/(4t)}}{(4\pi t)^{\dim L/2}}\,\mathrm dY.
		\]
		The integral is absolutely convergent. The holomorphic kernel has the signed affine lattice formula of \Cref{thm:heat-kernel-image}, and \Cref{thm:cartan-general} reduces the integral to the Cartan subalgebra $\tLie$.
	\end{theorem}
	
	The formula applies to every compact connected Lie group $G$ and incorporates its actual exponential lattice. Together with the affine lattice expression for $q_t^{\C}$, it gives an analytic inversion formula and a global image formula. On the heat kernel diagonal $g=e$, the zero vector in the lattice produces the nilpotent model, while the nonzero vectors are separated from the origin by the shortest exponential lattice length. Off the diagonal, the Varadhan--L\'eandre limit recovers the Carnot--Carath\'eodory (sub-Riemannian) distance. These two regimes lead to the geometric and spectral parts of the paper.
	
	Our second objective is to obtain an explicit small-time asymptotic estimate for the sub-Riemannian distance. Assume now that the pair is two-step,
	\[
	\gLie=\pLie+[\pLie,\pLie].
	\]
	In the present compact orthogonal setting, this condition is equivalent to $\Lie(\pLie)=\gLie$, as shown in \Cref{sec:preliminaries}. The tangent group at the identity is the two-step Carnot group on $\pLie\oplus\lLie$, and its homogeneous dimension is
	\[
	Q=m+2r=\dim G+\dim L;
	\]
	see, for example, \cite{AgrachevBarilariBoscain}. The Carnot-Carath\'eodory distance on $G$ has the global finite-dimensional description
	\[
	d_{\sr}(e,g)^2
	=\min\Bigl\{|B|^2:
	g=\exp(A+B)\exp(-A),\ A\in\lLie,\ B\in\pLie\Bigr\};
	\]
	see \Cref{prop:distance-general}. In the purely vertical direction, $g=\exp(\varepsilon Z)$ with $Z\in\lLie$, the second layer of the tangent group. As $\varepsilon\downarrow0$, the distance is of order $\varepsilon^{1/2}$, and the leading coefficient of its square is a positively homogeneous functional determined by the tangent group.
	
	For the quantitative vertical result we specialize to $L=K$. Then $\gLie=\kLie\oplus\pLie$ is the symmetric decomposition, and we assume $[\pLie,\pLie]=\kLie$. Let
	\[
	B:\Lambda^2\pLie\longrightarrow\kLie,
	\qquad B(X\wedge Y)=[X,Y].
	\]
	For $\Omega\in\Lambda^2\pLie$, let $\sigma_1(\Omega)\ge\cdots\ge \sigma_{\lfloor m/2\rfloor}(\Omega)\ge0$ be the singular values determined by the induced metric on $\pLie$, and define
	\[
	\mathfrak F_{G,K}(Z)
	=\min_{B\Omega=Z}
	4\pi\sum_{j=1}^{\lfloor m/2\rfloor}j\,\sigma_j(\Omega).
	\]
	By \Cref{thm:sp-free-distance}, the weighted singular-value functional $\mathfrak F_{G,K}$ is the exact vertical squared distance on the free two-step group; compare \cite[Theorem~2]{Brockett1982}. Passing through the central quotient induced by the bracket identifies it with the vertical distance on the tangent group of $(G,K)$. Separating the odd and even terms in the logarithmic expansion then compares the horizontal tangent loop with the compact endpoint $\exp(\varepsilon Z)$ in $G$.
	
	\begin{theorem}\label{thm:intro-vertical}
		For every $R>0$ there are $\varepsilon_R,C_R>0$ such that
		\[
		\left|d_{\sr}(e,\exp(\varepsilon Z))^2
		-\varepsilon\mathfrak F_{G,K}(Z)\right|
		\le C_R\varepsilon^{3/2}
		\]
		whenever $\mathfrak F_{G,K}(Z)\le R$ and $0<\varepsilon\le\varepsilon_R$. The lower estimate has the stronger form
		\[
		d_{\sr}(e,\exp(\varepsilon Z))^2
		\ge \varepsilon\mathfrak F_{G,K}(Z)-C_R\varepsilon^2.
		\]
	\end{theorem}
	
	The function $\mathfrak F_{G,K}$ defining the leading coefficient is attained, globally Lipschitz, and equivalent to the Euclidean norm on $\kLie$; see \Cref{prop:sp-regularity}. After choosing orthonormal bases, \eqref{eq:sp-finite-objective} expresses its value as the minimum of an explicit function of finitely many real variables, and the search can be restricted to a bounded cube. Evaluating this function on a finite grid gives the rigorous lower and upper bounds in \eqref{eq:sp-grid-enclosure}, whose gap is controlled by the grid spacing. This formula on the tangent free two-step nilpotent Lie group gives the weighted singular-value solution of the Gaveau-Brockett problem \cite{Gaveau,Brockett1982}; Li and Zhang established an exact distance description for arbitrary two-step groups \cite{LiZhang2021}.
	
	The diagonal specialization determines the local spectral invariants. Our third objective is to obtain a leading-term asymptotic for the heat kernel in terms of local invariants.
	\begin{theorem}
		For every compact connected two-step pair there are $c>0$ and $N\ge0$ such that, as $t\downarrow0$,
		\[
		p_t(e)=C_{G,L}t^{-Q/2}e^{\beta_{G,L}t}
		+O(t^{-N}e^{-c/t}).
		\]
		Moreover, with the orthonormal bases $Y_a$ of $\lLie$ and $X_i$ of $\pLie$ fixed in \Cref{sec:preliminaries}, and with all norms induced by the fixed invariant inner product,
		\[
		\beta_{G,L}
		=\frac{\operatorname{Scal}_G-\operatorname{Scal}_L}{6}
		=\frac1{12}\sum_{a,i}|[Y_a,X_i]|^2
		+\frac1{24}\sum_{i,j}|[X_i,X_j]|^2,
		\]
		and $C_{G,L}$ admits the root-theoretic integral and the absolutely convergent positive multiple Dirichlet series of \Cref{zeta:complete-local-constants}. Consequently the formulas determine every local heat coefficient, every zeta residue, and every nonpositive zeta value.
	\end{theorem}
	
	The leading term agrees with the general nilpotentization theorem of Colin de Verdi\`ere, Hillairet, and Tr\'elat \cite{ColinHillairetTrelat2021}. Compact Lie pairs lie between two extreme cases. For a general filtered manifold,
	\[
	p_t(x,x)\sim t^{-Q/2}\sum_{j\ge0}a_j(x)t^{j/2},
	\]
	with a sequence of local coefficients \cite{DaveHaller}. On a compact quotient of a graded nilpotent group, on the other hand, homogeneity gives only
	\[
	p_t^0(e)=t^{-Q/2}p_1^0(e),
	\]
	by a general result of Fischer \cite{Fischer2022}. In contact, quaternionic contact, and H-type settings, the first heat coefficients have been related explicitly to sub-Riemannian curvature invariants \cite{Barilari2013,Laaroussi2021,BauerMarkinaLaaroussiVegaMolino2026}. Compact Lie pairs lie between these cases:
	\[
	p_t(e)
	=C_{G,L}t^{-Q/2}
	\sum_{\ell\ge0}\frac{\beta_{G,L}^{\ell}}{\ell!}t^\ell
	+O(t^{-N}e^{-c/t}).
	\]
	Thus higher algebraic terms may occur, unlike in the nilpotent model, but they are all determined by the two constants $C_{G,L}$ and $\beta_{G,L}$, unlike on a general filtered manifold; see \eqref{zeta:ambient-small-time-estimate} and \eqref{zeta:ambient-small-time-bound}.
	
	As the last objective of this article, we study the zeta-regularized determinant. Set
	\[
	\zeta_{\pLie}(s)=\Tr'((-\Delta_{\pLie})^{-s}),
	\qquad
	\det\nolimits_\zeta'(-\Delta_{\pLie})=e^{-\zeta_{\pLie}'(0)}.
	\]
	
	\begin{theorem}
		Assume that $(G,L)$ is a two-step pair and that $G$ is simply connected. Then, for every $0<\tau\le1$,
		\[
		\log\det\nolimits_\zeta'(-\Delta_{\pLie})
		=\gamma_{\mathrm E}+\log\tau
		-C_{G,L}\mathscr A_{Q/2,\tau}(\beta_{G,L})
		-\mathcal R_{\mathrm{lat}}(\tau)
		-\mathcal R_{\mathrm{spec}}(\tau).
		\]
		The local series $\mathscr A_{Q/2,\tau}$ depends only on $Q/2$, $\tau$, and $\beta_{G,L}$ and is given explicitly in \Cref{zeta:section}. The terms $\mathcal R_{\mathrm{lat}}$ and $\mathcal R_{\mathrm{spec}}$ are absolutely convergent sums over nonzero coroot lattice orbits and positive branching eigenvalues, respectively. The expression is independent of $\tau$ and satisfies the exponential truncation bounds in \eqref{zeta:lattice-tail} and \eqref{zeta:spectral-tail}.
	\end{theorem}
	The two remainders are given exactly by \eqref{zeta:lattice-determinant-term} and \eqref{zeta:spectral-determinant-term}. The formula separates the determinant into one local term and two absolutely convergent global terms. The local term $-C_{G,L}\mathscr A_{Q/2,\tau}(\beta_{G,L})$ depends on the nilpotent tangent model through $C_{G,L}$ and on the scalar curvature difference through $\beta_{G,L}$. The global part is split in the manner of Ewald \cite{Ewald1921}: $-\mathcal R_{\mathrm{lat}}(\tau)$ is the sum over nonzero exponential lattice orbits, while $-\mathcal R_{\mathrm{spec}}(\tau)$ is the sum over the positive spectrum with its branching multiplicities. The term $\gamma_{\mathrm E}+\log\tau$ comes from removing the zero mode.
	
	This splitting is analogous to the classical Ewald expansions of Epstein zeta functions \cite{ChakrabortyKanemitsuTsukada2016}. Related formulas occur for products with a circle and for hyperbolic surfaces \cite{BauerFurutani2010,StrohmaierUski2013}. Here the interval $(0,\tau]$ is evaluated using the compact group exponential lattice formula, while $[\tau,\infty)$ is evaluated using the Peter-Weyl branching expansion. The two tail estimates quantify the truncation error and give the balancing choice of $\tau$ in \eqref{zeta:balanced-splitting-time}.
	
	We specialize these results to block subgroups of $\SU(N)$. This family includes semisimple block subgroups, full block Levi subgroups, and connected subtori of the block scalar factor. In the two-block full Levi case one recovers the compact symmetric pairs of type AIII\@. For the one-column family
	\[
	(G,K)=\bigl(\SU(n+1),S(\U(n)\times\U(1))\bigr),
	\]
	the vertical coefficient becomes the weighted eigenvalue functional $\Phi_n$ in \eqref{eq:sp-Phi}, while the two local spectral constants reduce to
	\[
	C_{G,K}=\left(\frac\pi2\right)^{n(n+1)},
	\qquad
	\beta_{G,K}=\frac{n(n+1)}4.
	\]
	This geometry also occurs in time-optimal control of multilevel quantum systems, where the horizontal controls couple all levels to a distinguished one \cite{BoscainChambrionGauthier2002,AlbertiniDAlessandroSheller2020}. The coefficient is attained by the explicit minimizing loop in \eqref{eq:sp-AIII-minimizer}. The right-$\SU(n)$-invariant sector recovers the standard CR sphere spectrum of Baudoin-Wang, while the rank-one member recovers the canonical contact spectrum on $\SU(2)$ used by Baudoin-Bonnefont and Bauer-Furutani; see the final example in \Cref{sec:examples}. 
	
	The paper is organized as follows. \Cref{sec:preliminaries} establishes the heat kernel inversion formula, proves the signed affine lattice expression for the holomorphic compact group heat kernel, and reduces the integral to a Cartan subalgebra. \Cref{sec:distance} derives the global distance minimum and identifies the two-step tangent group. \Cref{sec:sp-vertical} solves the free vertical problem, descends it through the bracket quotient, and proves the uniform compact vertical expansion. \Cref{zeta:section} determines the local heat and zeta data and the zeta-regularized determinant. \Cref{sec:examples} treats the concrete examples and their reductions to previously known cases.
	
	\section{Preliminaries and heat kernel representations}
	\label{sec:preliminaries}
	
	Throughout, $G$, $K$, and $L\subset K$ are as in the Introduction, and $\langle\cdot,\cdot\rangle$ is the fixed invariant inner product on $\gLie$. Let $\lLie=\Lie(L)$ and let $\pLie=\lLie^\perp$ be its orthogonal complement, so that
	\[
	\gLie=\lLie\oplus\pLie.
	\]
	Choose orthonormal bases
	\[
	Y_1,\dots,Y_r \text{ of }\lLie,
	\qquad
	X_1,\dots,X_m \text{ of }\pLie.
	\]
	For $X\in\gLie$, let $\widetilde X$ denote the left-invariant vector field
	\[
	(\widetilde X f)(g):=\left.\frac{d}{ds}\right|_{s=0}f(g e^{sX}).
	\]
	Define
	\[
	\Delta_{\lLie}:=\sum_{a=1}^{r}\widetilde Y_a^{\,2},
	\qquad
	\Delta_{\pLie}:=\sum_{i=1}^{m}\widetilde X_i^{\,2},
	\qquad
	\Delta_G:=\Delta_{\lLie}+\Delta_{\pLie}.
	\]
	Thus $\Delta_G$ is the bi-invariant Laplacian for the chosen metric. Haar measures on compact groups have total mass one, Lie algebras carry the Lebesgue measure induced by the fixed inner product, and $\vol(G)$ denotes the metric Riemannian volume of $G$.
	
	Fix a maximal torus $T\subset L$ with Lie algebra $\tLie$. We use real compact roots: for $H\in\tLie$, a complex root restricting as $\beta(H)=i\alpha(H)$ is represented by $\alpha\in\tLie^*$. For a positive system $R_L^+$, put
	\[
	\rho_L:=\frac12\sum_{\alpha\in R_L^+}\alpha,
	\qquad
	j_L(H):=\det\!\left(\frac{1-e^{-\operatorname{ad}H}}{\operatorname{ad}H}\right),
	\qquad
	j_L^{\nc}(Y):=j_L(iY).
	\]
	For $Y\in\lLie$, $j_L^{\nc}(Y)>0$; its square root below is the positive analytic one, normalized to one at the origin. The metric on $\tLie$ supplies the norm on $\tLie^*$.
	
	For every compact connected group, the heat kernel image formula is indexed by the actual exponential lattice. Choose a maximal torus $T_G\subset G$ containing $T$, with Lie algebra $\tLie_G\supset\tLie$, Weyl group $\Wt_G$, and positive real roots $R_G^+\subset\tLie_G^*$ in the same convention. Extend the metric complex bilinearly and write $\tLie_{G,\C}=\tLie_G\otimes_\R\C$; every root is extended complex linearly to $\tLie_{G,\C}$. Choose a faithful unitary representation $G\to\U(N)$ and identify $G$ with its image. Let $G_\C$ be the connected complex subgroup of $\operatorname{GL}(N,\C)$ generated by $G$ and $\exp(i\gLie)$, and let $L_\C$ be the connected complex subgroup generated by $L$ and $\exp(i\lLie)$. Then $L_\C\subset G_\C$, with Lie algebras $\gLie_\C$ and $\lLie_\C$, respectively. All holomorphic continuations below use these compatible matrix complexifications. Set
	\begin{equation*}
		\rho_G:=\frac12\sum_{\alpha\in R_G^+}\alpha,
		\qquad
		\Lambda_G:=\ker(\exp\colon\tLie_G\to T_G),
		\qquad
		\Pi_G(Z):=\prod_{\alpha\in R_G^+}\alpha(Z),
	\end{equation*}
	\begin{equation*}
		\mathscr S_G(Z):=
		\prod_{\alpha\in R_G^+}2\sin\!\left(\frac{\alpha(Z)}2\right),
		\qquad Z\in\tLie_{G,\C}.
	\end{equation*}
	Empty products equal $1$. In the same determinant convention as $j_L$, put
	\begin{equation*}
		j_G(X):=\det\!\left(\frac{1-e^{-\operatorname{ad}X}}
		{\operatorname{ad}X}\right),
		\qquad
		j_G^{\nc}(Y):=j_G(iY).
	\end{equation*}
	Thus the superscript $\nc$ stands for \emph{noncompact}: $j_G^{\nc}$ is the imaginary slice restriction of the analytic Jacobian, while $j_G$ denotes its real slice restriction. The notation $j_L^{\nc}$ has the corresponding meaning for $L$. The canonical analytic square root on the complexified Cartan subalgebra $\mathfrak t_{G,\C}$, normalized to equal $1$ at the origin, is
	\begin{equation}\label{eq:ambient-jacobian-square-root}
		j_G(Z)^{1/2}
		=\frac{\mathscr S_G(Z)}{\Pi_G(Z)}
		=\prod_{\alpha\in R_G^+}
		\frac{\sin(\alpha(Z)/2)}{\alpha(Z)/2},
	\end{equation}
	with the removable singularities understood by continuity. On the imaginary slice this convention gives the positive branch. We also write
	\begin{equation*}
		w(x):=\frac{x/2}{\sinh(x/2)},
		\qquad w(0):=1,
	\end{equation*}
	for the scalar reciprocal factor that occurs below.
	
	Since $\Delta_G$ is the Casimir operator and
	\[
	\Delta_{\pLie}=\Delta_G-\Delta_{\lLie},
	\]
	the operators $\Delta_G$, $\Delta_{\lLie}$, and $\Delta_{\pLie}$ commute on $C^\infty(G)$. The operator $\Delta_{\pLie}$ is the horizontal sub-Laplacian associated with $\pLie$.
	
	We say that $\pLie$ is \emph{bracket generating} when $\Lie(\pLie)=\gLie$, where $\Lie(\pLie)$ is the Lie subalgebra generated by $\pLie$. In the present setting, this is equivalent to
	\[
	\gLie=\pLie+[\pLie,\pLie].
	\]
	The reverse implication is immediate; conversely, metric invariance shows that any $Y\in\lLie$ orthogonal to $\pr_{\lLie}[\pLie,\pLie]$ is orthogonal to $\pLie$ and all its iterated brackets, hence vanishes when $\Lie(\pLie)=\gLie$. We call such a pair $(G,L)$ \emph{two-step}. Under this condition, the self-adjoint semigroup has a distribution kernel, and H\"ormander's hypoellipticity theorem \cite[Theorem~1.1]{Hormander1967} makes it smooth for $t>0$:
	\[
	p_t=e^{t\Delta_{\pLie}}\delta_e,
	\qquad t>0,
	\]
	and left invariance gives the right convolution formula
	\[
	(e^{t\Delta_{\pLie}}f)(g)
	=
	\int_G p_t(g^{-1}h)f(h)\,dh.
	\]
	Equivalently, the integral kernel is $K_t(g,h)=p_t(g^{-1}h)$.
	
	Let $q_t$ denote the ordinary heat kernel of $e^{t\Delta_G}$, normalized in the same way:
	\[
	q_t=e^{t\Delta_G}\delta_e.
	\]
	The bi-invariance of $\Delta_G$ makes $q_t$ central.
	
	Write $\Delta_L$ for the bi-invariant Laplacian on $L$ defined by the restricted metric. The compact group Segal-Bargmann transform gives, for every $t>0$, a unique holomorphic continuation
	\[
	q_t^{\C}:G_{\C}\to \C
	\]
	of the ordinary heat kernel $q_t$; see \cite[Theorem~1]{Hall1994}.
	
	For the semigroup $e^{t\Delta_L}$, Hall's inversion formula \cite[Theorem~1]{Hall1997} is used with $\Delta_L$. In these conventions it has the factor $e^{-|\rho_L|^2t}$ and the Gaussian density
	\[
	(4\pi t)^{-\dim L/2}e^{-|Y|^2/(4t)}.
	\]
	Explicitly, if $f\in C^\infty(L)$ and $F$ is the holomorphic continuation of $e^{t\Delta_L}f$ to $L_{\C}$, then
	\begin{equation}\label{eq:seg-bargmann-inversion}
		f(x)=e^{-|\rho_L|^2t}
		\int_{\lLie}F(x e^{iY})\,j_L^{\nc}(Y)^{1/2}
		\frac{e^{-|Y|^2/4t}}{(4\pi t)^{\dim L/2}}\,\mathrm{d}Y,
		\qquad x\in L.
	\end{equation}
	Hall's inversion theorem also gives the absolute convergence of this integral.
	
	\begin{lemma}\label{lem:strong-commutation}
		Let \(\mathscr P(G)\) be the algebraic Peter-Weyl sum, and set
		\[
		A=-\Delta_{\lLie},\qquad B=-\Delta_{\pLie},\qquad C=-\Delta_G
		\quad\hbox{on }\mathscr P(G).
		\]
		These operators are essentially self-adjoint, their nonnegative closures satisfy \(C=A+B\), and the closures of \(A\) and \(B\) strongly commute. Consequently,
		\begin{equation}\label{eq:strong-semigroup-factorization}
			e^{t\Delta_{\lLie}}e^{t\Delta_{\pLie}}=e^{t\Delta_G},
			\qquad t\ge0,
		\end{equation}
		and, in \(\mathcal D'(G)\),
		\begin{equation}\label{eq:kernel-semigroup-identity}
			e^{t\Delta_{\lLie}}p_t=q_t.
		\end{equation}
	\end{lemma}
	
	\begin{proof}
		Every Peter-Weyl summand is finite dimensional and invariant under \(A\), \(B\), and \(C\). On such a summand the Casimir \(C\) is scalar, \(A\) is Hermitian, and \(B=C-A\); hence the three restrictions have a common orthonormal eigenbasis. Taking the Hilbert direct sum of these finite-dimensional restrictions defines nonnegative self-adjoint operators \(\overline A\), \(\overline B\), and \(\overline C\). Truncation to finitely many Peter-Weyl summands converges in the graph norm of each operator, so \(\mathscr P(G)\) is a core and \(A,B,C\) are essentially self-adjoint. Their spectral projections are direct sums of commuting finite-dimensional projections; therefore \(\overline A\) and \(\overline B\) strongly commute and \(\overline C=\overline A+\overline B\) in their joint functional calculus. The joint spectral theorem \cite[Theorem~VIII.13]{ReedSimonI} now gives \eqref{eq:strong-semigroup-factorization}.
		
		Let \(\phi_n\in\mathscr P(G)\) be an approximate identity converging to \(\delta_e\) in \(\mathcal D'(G)\). Applying \eqref{eq:strong-semigroup-factorization} to \(\phi_n\), testing against a smooth function, and passing to the limit gives \(e^{t\Delta_{\lLie}}p_t=q_t\) in distributions, which is \eqref{eq:kernel-semigroup-identity}.
	\end{proof}
	
	\begin{theorem}\label{thm:integral-formula}
		Assume that $\pLie$ is bracket generating. Then, for every $g\in G$ and $t>0$,
		\[
		p_t(g)=e^{-|\rho_L|^2t}
		\int_{\lLie} q_t^{\C}(g e^{iY})\,j_L^{\nc}(Y)^{1/2}
		\frac{e^{-|Y|^2/4t}}{(4\pi t)^{\dim L/2}}\,\mathrm{d}Y.
		\]
		The integral is absolutely convergent.
	\end{theorem}
	
	\begin{proof}
		Fix $g\in G$ and $t>0$, and define
		\[
		f_{g,t}(l)=p_t(gl),
		\qquad l\in L.
		\]
		For the function $l\mapsto u(gl)$ and $Y\in\lLie$, the intrinsic left-invariant field on $L$ satisfies $\widetilde Y^L(u(g\,\cdot\,))(l)=(\widetilde Yu)(gl)$. Heat flow uniqueness on $L$ gives
		\[
		(e^{s\Delta_L}[u(g\,\cdot\,)])(l)
		=(e^{s\Delta_{\lLie}}u)(gl).
		\]
		Apply this with $u=p_t$ and $s=t$. By \Cref{lem:strong-commutation},
		\[
		e^{t\Delta_{\lLie}}p_t=q_t
		\quad\hbox{in }\mathcal D'(G).
		\]
		Since both sides are smooth for $t>0$, the identity holds pointwise. Consequently
		\[
		e^{t\Delta_L}f_{g,t}(l)=q_t(gl),
		\qquad l\in L.
		\]
		Let $F_{g,t}$ be the holomorphic continuation supplied by Hall's inversion theorem for $e^{t\Delta_L}f_{g,t}$ on $L_{\C}$. The inversion formula at $x=e$ gives
		\[
		p_t(g)=f_{g,t}(e)
		=
		e^{-|\rho_L|^2t}
		\int_{\lLie}F_{g,t}(e^{iY})\,j_L^{\nc}(Y)^{1/2}
		\frac{e^{-|Y|^2/4t}}{(4\pi t)^{\dim L/2}}\,\mathrm{d}Y.
		\]
		
		By the compatible matrix complexifications fixed above, $L_{\C}$ is a connected complex subgroup of $G_{\C}$. The map
		\[
		z\longmapsto q_t^{\C}(gz)
		\]
		is holomorphic on $L_{\C}$. On the real form $L$ it agrees with $e^{t\Delta_L}f_{g,t}$:
		\[
		q_t^{\C}(gl)=q_t(gl)=e^{t\Delta_L}f_{g,t}(l),
		\qquad l\in L.
		\]
		Since $L$ is a maximal totally real submanifold of the connected complex manifold $L_{\C}$, uniqueness of holomorphic continuation gives
		\[
		F_{g,t}(z)=q_t^{\C}(gz),
		\qquad z\in L_{\C}.
		\]
		In particular,
		\[
		F_{g,t}(e^{iY})=q_t^{\C}(g e^{iY}).
		\]
		Substitution proves the formula, and the absolute convergence is the corresponding conclusion of \eqref{eq:seg-bargmann-inversion}.
	\end{proof}
	
	The simply connected semisimple Poisson summation formula is classical. We record the compact connected form needed for the small-time remainder estimate \eqref{zeta:ambient-small-time-bound}. For $\gamma\in\Lambda_G$, every root satisfies $\alpha(\gamma)\in2\pi\Z$, and hence
	\begin{equation}\label{eq:ambient-lattice-sign}
		\varepsilon_G(\gamma):=e^{i\rho_G(\gamma)}\in\{1,-1\},
		\qquad
		\mathscr S_G(Z+\gamma)
		=\varepsilon_G(\gamma)\mathscr S_G(Z).
	\end{equation}
	For $t>0$ define the normally convergent theta series
	\begin{equation}\label{eq:ambient-theta}
		\Theta_{G,t}(Z)
		:=\sum_{\gamma\in\Lambda_G}
		\varepsilon_G(\gamma)\Pi_G(Z+\gamma)
		\exp\!\left(-\frac{\langle Z+\gamma,Z+\gamma\rangle}{4t}\right)
	\end{equation}
	and, initially away from the affine root hyperplanes,
	\begin{equation}\label{eq:ambient-W}
		\mathscr W_{G,t}(Z):=\frac{\Theta_{G,t}(Z)}{\mathscr S_G(Z)}.
	\end{equation}
	
	\begin{theorem}\label{thm:heat-kernel-image}
		For every compact connected $G$ and $t>0$, the apparent singularities in \eqref{eq:ambient-W} are removable. The resulting entire function on $\tLie_{G,\C}$ is invariant under $\Wt_G\ltimes\Lambda_G$, and
		\begin{equation}\label{eq:ambient-complex-image}
			q_t^{\C}(\exp Z)
			=\frac{\vol(G)e^{t|\rho_G|^2}}{(4\pi t)^{\dim G/2}}
			\mathscr W_{G,t}(Z),
			\qquad Z\in\tLie_{G,\C}.
		\end{equation}
		Its restriction to $Z\in\tLie_G$ is the ordinary heat kernel $q_t$.
	\end{theorem}
	
	\begin{proof}
		Let $Z_0=Z(G)_0$ and let $\pi_{\mathrm{sc}}\colon G_{\mathrm{sc}}\to[G,G]$ be the simply connected cover of the connected derived subgroup. Give $Z_0$ the restricted metric and $G_{\mathrm{sc}}$ the pulled-back metric. Metric invariance gives the orthogonal splitting
		\[
		\gLie=\mathfrak z(\gLie)\oplus[\gLie,\gLie].
		\]
		Thus
		\[
		\widetilde G:=Z_0\times G_{\mathrm{sc}},
		\qquad
		\pi\colon\widetilde G\longrightarrow G,
		\qquad
		\pi(z,\widetilde g)=z\,\pi_{\mathrm{sc}}(\widetilde g),
		\]
		is a finite central local isometry covering. Choose the product maximal torus $\widetilde T\subset\widetilde G$ above $T_G$ and identify their Lie algebras by $\mathrm d\pi$. Write $\widetilde\Lambda=\ker(\exp_{\widetilde G}|_{\Lie(\widetilde T)})$ and let $\widetilde q_t$ be the heat kernel on $\widetilde G$.
		
		Poisson summation on the torus factor and the wrapping formula on the simply connected compact semisimple factor \cite[Theorem~4.13]{Maher} give, for every regular $H\in\Lie(\widetilde T)$,
		\[
		\widetilde q_t(\exp_{\widetilde G}H)
		=\frac{\vol(\widetilde G)e^{t|\rho_G|^2}}
		{(4\pi t)^{\dim G/2}}
		\sum_{\widetilde\gamma\in\widetilde\Lambda}
		\frac{e^{-|H+\widetilde\gamma|^2/(4t)}}
		{j_G(H+\widetilde\gamma)^{1/2}}.
		\]
		Here the root system and Jacobian agree with those of $G$ under the Lie algebra identification. Let $F=\ker\pi$. The group $\widetilde G$ is connected and $F$ is finite and central, hence $F\subset\widetilde T$. Haar probability disintegration over the covering is
		\[
		\int_{\widetilde G}u(\widetilde g)\,\mathrm d\widetilde g
		=\frac1{|F|}\int_G
		\sum_{\pi(\widetilde g)=g}u(\widetilde g)\,\mathrm dg.
		\]
		Applying this identity to the heat semigroup, and then using uniqueness of holomorphic continuation, gives
		\[
		q_t^{G,\C}(\pi z)
		=\frac1{|F|}\sum_{f\in F}\widetilde q_t^{\,\C}(zf),
		\qquad
		\vol(\widetilde G)=|F|\vol(G).
		\]
		Under the common Lie algebra identification, the exponential lattices fit into the exact sequence
		\begin{equation}\label{eq:central-lattice-sequence}
			0\longrightarrow\widetilde\Lambda
			\longrightarrow\Lambda_G
			\xrightarrow{\ H\mapsto\exp_{\widetilde G}(H)\ }
			F\longrightarrow0.
		\end{equation}
		Indeed, $H\in\Lambda_G$ exponentiates in $\widetilde G$ to an element of $F$. The kernel is $\widetilde\Lambda$; surjectivity follows from $F\subset\widetilde T$ by taking a toral logarithm of each $f\in F$. Thus the sum over $F$ replaces $\widetilde\Lambda$ by its cosets in $\Lambda_G$. The factor $|F|^{-1}$ from Haar probability disintegration cancels the identity $\vol(\widetilde G)=|F|\vol(G)$, leaving the metric volume prefactor in \eqref{eq:ambient-complex-image}. Rewriting each reciprocal square root with \eqref{eq:ambient-jacobian-square-root} and \eqref{eq:ambient-lattice-sign} yields $\Theta_{G,t}(H)/\mathscr S_G(H)$, including the signs created by a mixed central quotient.
		
		This proves the real formula on the regular set. To pass across the singular set, let $c_t$ denote the prefactor in \eqref{eq:ambient-complex-image}. The function
		\[
		Z\longmapsto
		\mathscr S_G(Z)q_t^{\C}(\exp Z)-c_t\Theta_{G,t}(Z)
		\]
		is entire on $\tLie_{G,\C}$ and vanishes on the open regular subset of the maximal totally real space $\tLie_G$. Successive one-variable identity theorems, in real coordinates on $\tLie_G$, therefore show that it vanishes identically. On the regular set the quotient in \eqref{eq:ambient-W} equals $c_t^{-1}q_t^{\C}(\exp Z)$; the latter entire function defines its continuation across every affine root hyperplane. This continuation is invariant under both the Weyl group and translation by $\Lambda_G$. This proves \eqref{eq:ambient-complex-image} and all asserted continuation and invariance properties.
	\end{proof}
	
	\begin{corollary}\label{cor:heat-kernel-image-sc}
		If $G$ is simply connected, then $G$ is semisimple, $\Lambda_G=2\pi Q_G^\vee$, where $Q_G^\vee$ is the coroot lattice, and $\varepsilon_G\equiv1$. Consequently
		\begin{equation}\label{eq:ambient-complex-image-sc}
			q_t^{\C}(\exp Z)
			=\frac{\vol(G)e^{t|\rho_G|^2}}{(4\pi t)^{\dim G/2}}
			\frac{\displaystyle\sum_{\gamma\in\Lambda_G}
				\Pi_G(Z+\gamma)
				e^{-\langle Z+\gamma,Z+\gamma\rangle/4t}}
			{\mathscr S_G(Z)},
		\end{equation}
		where the quotient is understood through its entire continuation across the removable singularities.
	\end{corollary}
	Indeed, for $\gamma=2\pi\eta$ with $\eta\in Q_G^\vee$, the integrality of $\rho_G(\eta)$ makes the sign in \eqref{eq:ambient-lattice-sign} equal to one.
	
	The holomorphic continuation $q_t^{\C}$ is conjugation invariant. It is convenient to normalize the resulting class function by
	\begin{equation*}
		\mathbf W_{G,t}(x)
		:=\frac{(4\pi t)^{\dim G/2}}{\vol(G)e^{t|\rho_G|^2}}q_t^{\C}(x),
		\qquad x\in G_{\C}.
	\end{equation*}
	If a regular semisimple $x$ is conjugate to $\exp Z$, then $\mathbf W_{G,t}(x)=\mathscr W_{G,t}(Z)$; affine Weyl invariance makes this independent of the toral representative and logarithm branch. All other values are supplied by holomorphic continuation.
	
	Combining the image and integral formulas gives, whenever $\pLie$ is bracket generating,
	\begin{equation*}
		\begin{aligned}
			p_t(g)={}&
			\frac{\vol(G)e^{t(|\rho_G|^2-|\rho_L|^2)}}
			{(4\pi t)^{(\dim G+\dim L)/2}}
			\\
			&\times\int_{\lLie}
			\mathbf W_{G,t}(ge^{iY})j_L^{\nc}(Y)^{1/2}
			e^{-|Y|^2/4t}\,\mathrm dY.
		\end{aligned}
	\end{equation*}
	At every regular semisimple argument $ge^{iY}$, the factor $\mathbf W_{G,t}(ge^{iY})$ is the explicit signed theta quotient in \eqref{eq:ambient-theta} and \eqref{eq:ambient-W}, evaluated at any toral logarithm.
	
	The integral admits a further reduction to a Cartan subalgebra. For the maximal torus $T\subset L$ and positive roots fixed above, define
	\[
	\Pi_L(H):=\prod_{\alpha\in R_L^+}\alpha(H).
	\]
	Let $\Wt_L$ be the Weyl group. Haar measures on $L$ and $T$ have mass one, and $\mathrm d\dot k$ is the quotient probability measure determined by
	\[
	\int_L F(k)\,\mathrm{d}k
	=
	\int_{L/T}\int_T F(kt)\,\mathrm{d}t\,\mathrm{d}\dot k
	\]
	for continuous $F$. With the metric Euclidean measures, set
	\[
	C_L:=\frac{\vol(L)}{\vol(T)|\Wt_L|}.
	\]
	The Lie algebra Weyl integration formula is
	\begin{equation}\label{eq:weyl-lie}
		\int_{\lLie}F(Y)\,\mathrm{d}Y
		=
		C_L
		\int_{\tLie}\int_{L/T}F(\Ad(k)H)\,|\Pi_L(H)|^2
		\,\mathrm{d}\dot k\,\mathrm{d}H,
	\end{equation}
	for $F\in L^1(\lLie)$. Indeed, on the regular set the map $(kT,H)\mapsto\Ad(k)H$ from $L/T\times\tLie$ to $\lLie$ is $|\Wt_L|$-to-one and has normal Jacobian $|\Pi_L(H)|^2$. The quotient Riemannian volume is $(\vol(L)/\vol(T))\,\mathrm d\dot k$, so the area formula gives \eqref{eq:weyl-lie}; the singular sets have measure zero.
	
	For $H\in\tLie$, the root space determinant computation with the real compact root convention gives
	\[
	j_L^{\nc}(H)^{1/2}
	=
	\prod_{\alpha\in R_L^+}
	\frac{\sinh\!\bigl(\alpha(H)/2\bigr)}{\alpha(H)/2}.
	\]
	Indeed, on the real two-plane attached to $\alpha$ the eigenvalues of $\operatorname{ad}H$ are $\pm i\alpha(H)$, and replacing $H$ by $iH$ changes the usual sine factor into the displayed hyperbolic sine factor. The quotient is interpreted by its removable value $1$ when $\alpha(H)=0$.
	
	For $g\in G$, $t>0$, and $H\in\tLie$, define the orbital average
	\[
	\mathcal O_{g,t}(H)
	:=
	\int_{L/T} q_t^{\C}(gk e^{iH}k^{-1})\,\mathrm{d}\dot k.
	\]
	This is well defined because $e^{iH}$ is fixed by conjugation with $T$.
	
	\begin{theorem}\label{thm:cartan-general}
		Assume that $\pLie$ is bracket generating. Then, for every $g\in G$ and $t>0$,
		\[
		p_t(g)
		=
		C_L e^{-|\rho_L|^2t}
		\int_{\tLie}
		\mathcal O_{g,t}(H)\,
		j_L^{\nc}(H)^{1/2}|\Pi_L(H)|^2
		\frac{e^{-|H|^2/4t}}{(4\pi t)^{\dim L/2}}
		\,\mathrm{d}H.
		\]
	\end{theorem}
	
	\begin{proof}
		By \Cref{thm:integral-formula},
		\[
		p_t(g)=e^{-|\rho_L|^2t}\int_{\lLie}\Phi_{g,t}(Y)\,\mathrm{d}Y,
		\]
		where
		\[
		\Phi_{g,t}(Y)
		:=
		q_t^{\C}(g e^{iY})\,j_L^{\nc}(Y)^{1/2}
		\frac{e^{-|Y|^2/4t}}{(4\pi t)^{\dim L/2}}.
		\]
		The absolute convergence in \Cref{thm:integral-formula} gives $\Phi_{g,t}\in L^1(\lLie)$, so \eqref{eq:weyl-lie} applies.
		
		Since $j_L^{\nc}$ and $|Y|$ are $\Ad(L)$-invariant, for $H\in\tLie$ we have
		\[
		\Phi_{g,t}(\Ad(k)H)
		=
		q_t^{\C}(g e^{i\Ad(k)H})\,j_L^{\nc}(H)^{1/2}
		\frac{e^{-|H|^2/4t}}{(4\pi t)^{\dim L/2}}.
		\]
		Using
		\[
		e^{i\Ad(k)H}=k e^{iH}k^{-1},
		\]
		we obtain
		\[
		\int_{L/T}\Phi_{g,t}(\Ad(k)H)\,\mathrm{d}\dot k
		=
		\mathcal O_{g,t}(H)\,
		j_L^{\nc}(H)^{1/2}
		\frac{e^{-|H|^2/4t}}{(4\pi t)^{\dim L/2}}.
		\]
		Substituting this identity into the Lie algebra Weyl integration formula \eqref{eq:weyl-lie} proves the formula. Applying the same computation to $|\Phi_{g,t}|$ and using
		\[
		|\mathcal O_{g,t}(H)|
		\le
		\int_{L/T}|q_t^{\C}(gk e^{iH}k^{-1})|\,\mathrm{d}\dot k
		\]
		shows that the Cartan integrand is absolutely integrable on $\tLie$.
	\end{proof}
	
	At $g=e$, conjugation invariance gives $\mathcal O_{e,t}(H)=q_t^{\C}(e^{iH})$. On the dominant Weyl chamber $\tLie_+$, all factors in $\Pi_L(H)$ are nonnegative; hence \Cref{thm:heat-kernel-image,thm:cartan-general} give the diagonal formula
	\begin{equation*}
		\begin{aligned}
			p_t(e)={}&
			\frac{C_L|\Wt_L|\vol(G)e^{t(|\rho_G|^2-|\rho_L|^2)}}
			{(4\pi t)^{(\dim G+\dim L)/2}}
			\\
			&\times\int_{\tLie_+}
			\mathscr W_{G,t}(iH)\Pi_L(H)^2j_L^{\nc}(H)^{1/2}e^{-|H|^2/4t}
			\,\mathrm dH.
		\end{aligned}
	\end{equation*}
	
	\section{Two-step compact Lie pairs and the sub-Riemannian distance}
	\label{sec:distance}
	
	Throughout this section, $L$ remains an arbitrary closed connected subgroup of $K$ and $(G,L)$ is assumed to be two-step. The left-invariant horizontal distribution is $\mathcal H_g=(L_g)_*\pLie$, and its Carnot-Carath\'eodory distance is the distance $d_{\sr}$ defined in \eqref{eq:intro-sr-distance}.
	
	\begin{lemma}\label{lem:normal-geodesics}
		Every normal geodesic issuing from $e$ has the form
		\begin{equation}\label{eq:normal-geodesic}
			\gamma_{A,B}(t)=\exp\!\bigl(t(A+B)\bigr)\exp(-tA),
			\qquad A\in\lLie,\quad B\in\pLie.
		\end{equation}
		Its left-trivialized velocity is $\Ad_{e^{tA}}B$ and its speed is $|B|$. If $\gLie=\pLie+[\pLie,\pLie]$, every length minimizer has a normal lift.
	\end{lemma}
	
	\begin{proof}
		Identify $\gLie^*$ with $\gLie$ by the invariant inner product and write a costate as $A+B$ with $A\in\lLie$ and $B\in\pLie$. The normal Hamiltonian is $|B|^2/2$. Its Hamiltonian flow keeps $A$ constant and gives the control $u(t)=\Ad_{e^{tA}}B$. Integrating $\gamma^{-1}\dot\gamma=u$ yields \eqref{eq:normal-geodesic}; see \cite[Equation~(7.46)]{AgrachevBarilariBoscain}. For a two-step distribution, the Goh necessary condition for a strictly abnormal minimizer forces its covector to annihilate $\pLie+[\pLie,\pLie]=\gLie$. The covector is then zero, so every minimizer is normal; equivalently, apply \cite[Corollary~12.15]{AgrachevBarilariBoscain}.
	\end{proof}
	
	To distinguish the two geometries, we retain $d_{\sr}$ for the Carnot-Carath\'eodory distance on $G$ and reserve $d_0$ for the Carnot-Carath\'eodory distance on the nilpotent tangent group. For the two-step pair considered here, the nilpotentization at the identity is the simply connected group $\mathbb N_{G,L}$ on $\pLie\oplus\lLie$ with product
	\begin{equation}\label{eq:tangent-group-product}
		(X,Z)(X',Z')
		=\left(X+X',Z+Z'+\frac12\pr_{\lLie}[X,X']\right).
	\end{equation}
	Its first layer is $\pLie$, its central layer is $\lLie$, and its homogeneous dimension is $Q$. We write $d_0$ for its Carnot-Carath\'eodory distance and $p_t^0$ for the heat kernel of $\exp(t\sum_{i=1}^m\widetilde X_i^{\,2})$ with respect to the Lebesgue measure in exponential coordinates induced by the metric. In the symmetric specialization $L=K$, one has $\lLie=\kLie$ and $[\pLie,\pLie]\subset\kLie$, so the projection in \eqref{eq:tangent-group-product} may be omitted.
	
	\begin{proposition}\label{prop:distance-general}
		Assume $(G,L)$ is two-step. Then, for every $g\in G$,
		\[
		d_{\sr}(e,g)^2
		=
		\min\Bigl\{|B|^2:\ g=\exp(A+B)\exp(-A),\ A\in\lLie,\ B\in\pLie\Bigr\}.
		\]
	\end{proposition}
	
	\begin{proof}
		The two-step condition implies that the horizontal distribution is bracket generating. Since $G$ is compact, its Carnot-Carath\'eodory metric is complete and proper, and therefore there is a length-minimizing horizontal curve $\gamma$ from $e$ to $g$. By \Cref{lem:normal-geodesics}, every minimizer has a normal lift. After constant-speed parametrization, \eqref{eq:normal-geodesic} gives
		\[
		\gamma(t)=\exp\!\bigl(t(A+B)\bigr)\exp(-tA)
		\]
		with $\gamma(1)=g$ and $|B|=d_{\sr}(e,g)$. This proves attainment and one inequality. Conversely, every pair satisfying $g=\exp(A+B)\exp(-A)$ supplies the horizontal curve $\gamma_{A,B}$ of length $|B|$, proving the reverse inequality. For arbitrary initial and terminal points $g_1,g_2\in G$, left invariance replaces $g$ by $g_1^{-1}g_2$.
	\end{proof}
	
	When $\pLie$ is bracket generating, $G$ is complete and equiregular. Moreover, the left-invariant horizontal fields have zero divergence with respect to Haar measure because $G$ is unimodular. Thus the sum of squares is the intrinsic hypoelliptic Laplacian associated with this sub-Riemannian structure, in the sense of~\cite{AgrachevBoscainGauthierRossi2009}, and L\'eandre's logarithmic formula~\cite{BarilariBoscainNeel2012} gives, for the present generator,
	\begin{equation}\label{eq:leandre-logarithmic}
		\lim_{t\downarrow0}4t\log p_t(g)=-d_{\sr}(e,g)^2.
	\end{equation}
	
	\section{Vertical asymptotics for compact symmetric pairs}
	\label{sec:sp-vertical}
	
	All notation remains as in the Introduction; in this section, we set $L=K$. Thus $K=(G^\theta)_0$ for the fixed involution $\theta$, and the standing $\Ad(G)$-invariant inner product is also $\theta$-invariant. The differential of $\theta$ gives the orthogonal decomposition
	\[
	\gLie=\kLie\oplus\pLie
	\]
	into the $+1$ and $-1$ eigenspaces. Hence
	\begin{equation}\label{eq:sp-brackets}
		[\kLie,\kLie]\subset\kLie,
		\qquad
		[\kLie,\pLie]\subset\pLie,
		\qquad
		[\pLie,\pLie]\subset\kLie.
	\end{equation}
	Assume
	\begin{equation}\label{eq:sp-bracket-surjective}
		[\pLie,\pLie]=\kLie.
	\end{equation}
	This is precisely the bracket-generating condition for the symmetric decomposition.
	
	The proof of \Cref{thm:intro-vertical} begins with the finite singular-value formula \eqref{eq:sp-F} and its finite-mesh bounds \eqref{eq:sp-grid-enclosure}. We then identify its objective with the vertical distance on the free two-step group in \Cref{thm:sp-free-distance}, descend that identity through the central quotient in \Cref{prop:sp-central-quotient}, and obtain the tangent distance in \Cref{thm:sp-tangent-distance}. Finally, \Cref{thm:sp-compact-vertical} compares the endpoint of the same scaled horizontal loop in the tangent group with the compact endpoint $\exp(\varepsilon Z)$ in $G$ and yields the vertical asymptotic stated in the Introduction.
	
	First consider the finite vertical coefficient. Equip $\Lambda^2\pLie$ with the Euclidean structure induced by the fixed inner product on $\pLie$. All singular values below are computed with respect to this metric. For $\Omega\in\Lambda^2\pLie$, define $J_\Omega\in\mathfrak{so}(\pLie)$ by
	\begin{equation}\label{eq:sp-J}
		\langle J_\Omega X,Y\rangle
		=\langle\Omega,X\wedge Y\rangle,
		\qquad X,Y\in\pLie.
	\end{equation}
	Write $\kappa:=\lfloor m/2\rfloor$ and let
	\[
	\sigma_1(\Omega)\ge\cdots\ge\sigma_\kappa(\Omega)\ge0
	\]
	be the nonnegative normal form coefficients of the skew-symmetric operator $J_\Omega$. Each occurs twice among the singular values of $J_\Omega$, with one additional zero when $m$ is odd. Thus there are orthonormal vectors $u_1,v_1,\ldots,u_\kappa,v_\kappa$ such that
	\begin{equation}\label{eq:sp-skew-normal}
		\Omega=\sum_{j=1}^{\kappa}\sigma_j(\Omega)u_j\wedge v_j.
	\end{equation}
	Set
	\begin{equation*}
		\Psi_{\pLie}(\Omega)
		:=4\pi\sum_{j=1}^{\kappa}j\,\sigma_j(\Omega).
	\end{equation*}
	
	The bracket defines the surjective linear map
	\begin{equation}\label{eq:sp-B}
		B:\Lambda^2\pLie\longrightarrow\kLie,
		\qquad
		B(X\wedge Y):=[X,Y].
	\end{equation}
	If $B^*:\kLie\to\Lambda^2\pLie$ is its metric adjoint, invariance of the inner product gives the identity
	\begin{equation}\label{eq:sp-isotropy-bracket-bridge}
		J_{B^*Y}=\operatorname{ad}(Y)|_{\pLie},
		\qquad Y\in\kLie.
	\end{equation}
	For $Z\in\kLie$, define
	\begin{equation}\label{eq:sp-F}
		\mathfrak F_{G,K}(Z)
		:=
		\min\left\{
		\Psi_{\pLie}(\Omega):
		\Omega\in\Lambda^2\pLie,\ B\Omega=Z
		\right\}.
	\end{equation}
	The minimum is attained because
	\begin{equation}\label{eq:sp-coercive}
		\Psi_{\pLie}(\Omega)
		\ge4\pi
		\left(\sum_{j=1}^{\kappa}\sigma_j(\Omega)^2\right)^{1/2}
		=4\pi|\Omega|
	\end{equation}
	and $B^{-1}(Z)$ is closed. Equivalently,
	\begin{equation*}
		\mathfrak F_{G,K}(Z)
		=4\pi\min\left\{
		\sum_{j=1}^{\kappa}j\sigma_j:
		\begin{array}{l}
			\sigma_1\ge\cdots\ge\sigma_\kappa\ge0,\\
			u_1,v_1,\ldots,u_\kappa,v_\kappa\text{ are orthonormal},\\
			Z=\displaystyle\sum_{j=1}^{\kappa}\sigma_j[u_j,v_j]
		\end{array}
		\right\}.
	\end{equation*}
	
	\begin{proposition}\label{prop:sp-regularity}
		The function $\mathfrak F_{G,K}$ is positively homogeneous of degree one, globally Lipschitz, and equivalent to the Euclidean norm on $\kLie$. Consequently there are constants $c,C,C_{\mathrm{Lip}}>0$, depending only on the pair, such that
		\begin{equation*}
			c|Z|\le\mathfrak F_{G,K}(Z)\le C|Z|,
			\qquad
			|\mathfrak F_{G,K}(Z)-\mathfrak F_{G,K}(Z')|
			\le C_{\mathrm{Lip}}|Z-Z'|.
		\end{equation*}
	\end{proposition}
	
	\begin{proof}
		The Hoffman-Wielandt inequality states that, for normal matrices $A,C$,
		\[
		\min_{\pi\in\mathfrak S_m}
		\sum_{j=1}^{m}|\lambda_j(A)-\lambda_{\pi(j)}(C)|^2
		\le \|A-C\|_{\mathrm{HS}}^2,
		\]
		where $\|A\|_{\mathrm{HS}}^2=\Tr(A^*A)$; see \cite[Theorem~1]{HoffmanWielandt}. For a real skew-symmetric operator, the eigenvalues are the pairs $\pm i\sigma_j$, together with a possible zero. Ordering them by their imaginary parts pairs $i\sigma_j(\Omega)$ with $i\sigma_j(\Omega')$ and the corresponding negative eigenvalues. Hence
		\[
		2\sum_j(\sigma_j(\Omega)-\sigma_j(\Omega'))^2
		\le\|J_\Omega-J_{\Omega'}\|_{\mathrm{HS}}^2.
		\]
		Since $\|J_\Omega-J_{\Omega'}\|_{\mathrm{HS}}^2=2|\Omega-\Omega'|^2$, it follows that
		\begin{equation}\label{eq:sp-Psi-Lipschitz}
			|\Psi_{\pLie}(\Omega)-\Psi_{\pLie}(\Omega')|
			\le4\pi\left(\sum_{j=1}^{\kappa}j^2\right)^{1/2}
			|\Omega-\Omega'|.
		\end{equation}
		Choose a linear right inverse $R:\kLie\to\Lambda^2\pLie$ of $B$. If $\Omega_Z$ minimizes \eqref{eq:sp-F}, then $\Omega_Z+R(Z'-Z)$ is feasible at $Z'$, because
		\[
		B\bigl(\Omega_Z+R(Z'-Z)\bigr)=Z'.
		\]
		Comparison in both directions gives the Lipschitz estimate. Homogeneity is immediate, and $RZ$ gives the upper norm bound. Conversely, $|Z|\le\|B\||\Omega|$ together with \eqref{eq:sp-coercive} gives the lower bound after taking the infimum.
	\end{proof}
	
	For computation, we first express the affine constraint $B\Omega=Z$ in orthonormal coordinates. Eliminating this constraint reduces the coefficient to a finite-dimensional unconstrained minimization; a mesh of the bounded search cube then gives lower and upper bounds.
	
	Choose orthonormal bases
	\[
	X_1,\ldots,X_m\quad\text{of }\pLie,
	\qquad
	Y_1,\ldots,Y_r\quad\text{of }\kLie,
	\]
	and order the pairs $(i,j)$, $1\le i<j\le m$, lexicographically. For $\omega\in\mathbb R^{\binom m2}$, write
	\[
	\Omega(\omega)
	:=\sum_{i<j}\omega_{ij}X_i\wedge X_j.
	\]
	The bracket map and target have the concrete coordinates
	\begin{equation*}
		\mathbf B_{a,(i,j)}
		:=\langle Y_a,[X_i,X_j]\rangle,
		\qquad
		z_a:=\langle Y_a,Z\rangle,
		\qquad Z\in\kLie.
	\end{equation*}
	Thus
	\begin{equation*}
		\mathfrak F_{G,K}(Z)
		=
		\min_{\mathbf B\omega=z}
		4\pi\sum_{j=1}^{\kappa}j\,\sigma_j(\Omega(\omega)),
	\end{equation*}
	where $\mathbf B\in\mathbb R^{r\times\binom m2}$ has full row rank.
	
	Let $\omega_0$ be the minimum-norm solution of $\mathbf B\omega=z$, and let the columns of $N\in\mathbb R^{\binom m2\times(\binom m2-r)}$ be an orthonormal basis of $\ker\mathbf B$. Every feasible point is uniquely $\omega=\omega_0+Ny$, and hence
	\begin{equation}\label{eq:sp-finite-objective}
		\mathfrak F_{G,K}(Z)
		=\min_{y\in\mathbb R^{\binom m2-r}}
		\Psi_{\pLie}\bigl(\Omega(\omega_0+Ny)\bigr).
	\end{equation}
	Orthogonality of $\omega_0$ and $\ker\mathbf B$, together with \eqref{eq:sp-coercive}, shows that a minimizer $y^*\in\mathbb R^{\binom m2-r}$ lies in the cube
	\[
	\left[
	-\frac{\Psi_{\pLie}(\Omega(\omega_0))}{4\pi},
	\frac{\Psi_{\pLie}(\Omega(\omega_0))}{4\pi}
	\right]^{\binom m2-r}.
	\]
	Evaluate the objective on any finite mesh of this cube whose coordinate spacing is at most $\eta$, and denote the smallest sampled value by $M_\eta$. The covering radius is at most $\tfrac12\sqrt{\binom m2-r}\,\eta$, so \eqref{eq:sp-Psi-Lipschitz} gives
	\begin{equation}\label{eq:sp-grid-enclosure}
		\max\left\{0,\,
		M_\eta
		-2\pi\left(\sum_{j=1}^{\kappa}j^2\right)^{1/2}
		\sqrt{\binom m2-r}\,\eta\right\}
		\le\mathfrak F_{G,K}(Z)\le M_\eta.
	\end{equation}
	If $\binom m2=r$, the feasible point $\omega_0$ is unique. Otherwise the mesh cardinality grows exponentially with $\dim\ker\mathbf B$.
	
	This finite coefficient is the vertical distance in the free two-step model. Let $\mathbb F_2(\pLie)$ be the free two-step Carnot group with Lie algebra
	\[
	\pLie\oplus\Lambda^2\pLie,
	\qquad
	[X,Y]_{\mathrm{free}}=X\wedge Y.
	\]
	In exponential coordinates its product is
	\begin{equation*}
		(X,\Omega)(Y,\Omega')
		=
		\left(X+Y,\Omega+\Omega'+\frac12X\wedge Y\right).
	\end{equation*}
	Give the first layer the original inner product and set
	\[
	\Delta_{\mathrm{free}}
	:=\sum_{i=1}^m(\widetilde X_i^{\mathrm{free}})^2,
	\]
	the horizontal sub-Laplacian formed from the corresponding left-invariant fields. Let $q_t^{\mathrm{free}}$ be its heat kernel relative to the Lebesgue measure in exponential coordinates. The homogeneous dimension is $m^2$. The vertical distance identity below goes back to Brockett \cite[Theorem~2]{Brockett1982}; we give a heat kernel proof in the present normalization.
	
	\begin{lemma}\label{lem:sp-free-radial}
		Identify $\Lambda^2\pLie$ with $\mathfrak{so}(\pLie)$ by \eqref{eq:sp-J} and equip $\mathfrak{so}(\pLie)$ with the metric $\langle A,C\rangle=-\tfrac12\Tr(AC)$. We use the Lebesgue measure induced by this metric on $\mathfrak{so}(\pLie)$ and the normalized Haar measure on $\mathrm O(m)$. Let $\Omega$ be regular, with $\sigma_1>\cdots>\sigma_\kappa>0$. There are nonzero constants $c_{2\kappa}^{\circ}$ and $c_{2\kappa+1}^{\circ}$, fixed by these measure normalizations, such that
		\begin{align}
			q_1^{\mathrm{free}}(0,\Omega)
			&=
			\frac{c_{2\kappa}^{\circ}}
			{\displaystyle\prod_{j<k}(\sigma_j^2-\sigma_k^2)}
			\det\!\left[\widehat w^{(2j-2)}(\sigma_k/2)\right]_{j,k=1}^{\kappa},
			&&m=2\kappa,
			\label{eq:sp-free-radial-even}\\
			q_1^{\mathrm{free}}(0,\Omega)
			&=
			\frac{c_{2\kappa+1}^{\circ}}
			{\displaystyle(\prod_j\sigma_j)
				\prod_{j<k}(\sigma_j^2-\sigma_k^2)}
			\det\!\left[\widehat w^{(2j-1)}(\sigma_k/2)\right]_{j,k=1}^{\kappa},
			&&m=2\kappa+1,
			\label{eq:sp-free-radial-odd}
		\end{align}
		where $w(a)=(a/2)/\sinh(a/2)$ is the function defined in \Cref{sec:preliminaries}, and
		\begin{equation}\label{eq:sp-scalar-transform}
			\widehat w(x):=\int_{\mathbb R}e^{-iax}w(a)\,\mathrm da
			=\pi^2\operatorname{sech}^2(\pi x).
		\end{equation}
	\end{lemma}
	
	\begin{proof}
		Fourier transformation in the center and Mehler's formula give
		\begin{equation}\label{eq:sp-free-central-Fourier}
			q_1^{\mathrm{free}}(0,\Omega)
			=c_m\int_{\mathfrak{so}(\pLie)}
			e^{-i\langle A,\Omega\rangle/2}
			\prod_{j=1}^{\kappa}w(a_j(A))\,\mathrm dA,
		\end{equation}
		where $a_j=a_j(A)$ are the positive singular values of $A\in\mathfrak{so}(\pLie)$ and $c_m>0$ is the Fourier normalization; see \cite[formula~(4.9$'$)]{Cygan1979}. The integrand apart from the phase is invariant under the full orthogonal group.
		
		Choose the standard skew Cartan with $2$-by-$2$ blocks $a_j\bigl(\begin{smallmatrix}0&1\\-1&0\end{smallmatrix}\bigr)$. The metric Weyl integration formula has Jacobian
		\[
		\prod_{j<k}(a_j^2-a_k^2)^2
		\quad(m=2\kappa),
		\qquad
		\prod_j a_j^2\prod_{j<k}(a_j^2-a_k^2)^2
		\quad(m=2\kappa+1).
		\]
		Put
		\[
		\Delta(a^2):=\prod_{j<k}(a_j^2-a_k^2),
		\qquad
		\Delta(\sigma^2):=\prod_{j<k}(\sigma_j^2-\sigma_k^2).
		\]
		For Cartan representatives $A(a)$ and $\Omega(\sigma)$, the orthogonal orbital integral formulas with Haar probability measure read
		\begin{align*}
			\int_{\mathrm O(2\kappa)}
			e^{-i\langle A(a),U\Omega(\sigma)U^{-1}\rangle/2}\,\mathrm dU
			&=C_{2\kappa}^{\mathrm{orb}}
			\frac{\det[\cos(a_j\sigma_k/2)]_{j,k=1}^{\kappa}}
			{\Delta(a^2)\Delta(\sigma^2)},\\
			\int_{\mathrm O(2\kappa+1)}
			e^{-i\langle A(a),U\Omega(\sigma)U^{-1}\rangle/2}\,\mathrm dU
			&=C_{2\kappa+1}^{\mathrm{orb}}
			\frac{\det[\sin(a_j\sigma_k/2)]_{j,k=1}^{\kappa}}
			{(\prod_j a_j\sigma_j)\Delta(a^2)\Delta(\sigma^2)}.
		\end{align*}
		These are the orthogonal group formulas in \cite[Equations~(4.16) and~(4.19)]{McSwiggen2021}, analytically continued from the hyperbolic to the trigonometric form and rescaled from the Hilbert-Schmidt Cartan coordinates to $\langle A,C\rangle=-\tfrac12\Tr(AC)$. The phase in \eqref{eq:sp-free-central-Fourier} then fixes the arguments $a_j\sigma_k/2$. With the displayed discriminant convention and $K_\kappa:=\kappa(\kappa-1)/2$, normalization at $A=0$ gives
		\begin{align*}
			C_{2\kappa}^{\mathrm{orb}}
			&=(-1)^{K_\kappa}2^{\kappa(\kappa-1)}
			\prod_{r=0}^{\kappa-1}(2r)!,\\
			C_{2\kappa+1}^{\mathrm{orb}}
			&=(-1)^{K_\kappa}2^{\kappa^2}
			\prod_{r=0}^{\kappa-1}(2r+1)!.
		\end{align*}
		Indeed, divide the Taylor determinants at $A=0$ by $\Delta(a^2)\Delta(\sigma^2)$; their leading coefficients are respectively $(-1)^{K_\kappa}2^{-\kappa(\kappa-1)}/ \prod_{r=0}^{\kappa-1}(2r)!$ and $(-1)^{K_\kappa}2^{-\kappa^2}/ \prod_{r=0}^{\kappa-1}(2r+1)!$. In even dimension the average is over $\mathrm O(2\kappa)$. Its two connected components exchange the sign of the Pfaffian, so the Pfaffian-sensitive term in the $\mathrm{SO}(2\kappa)$ formula cancels and the cosine determinant remains.
		
		One Weyl discriminant cancels in each case. Andr\'eief's identity for a general measure \cite[Proposition~7.1]{Zygouras2022}, applied with $\mathrm d\mu(a)=w(a)\,\mathrm da$, converts the remaining discriminant and the sine or cosine determinant into the derivative determinants in \eqref{eq:sp-free-radial-even}-\eqref{eq:sp-free-radial-odd}. The metric Weyl constant and the Fourier factor $c_m$ are positive, the orbital constants displayed above are nonzero, and the passage from the Weyl chamber to $\mathbb R^\kappa$ together with Andr\'eief's identity contributes fixed nonzero factors. Their product is the finite nonzero constant $c_m^{\circ}$ in the statement. Finally, residue summation at $2\pi i\mathbb Z\setminus\{0\}$ gives \eqref{eq:sp-scalar-transform}, including its value at zero by continuity.
	\end{proof}
	
	\begin{lemma}\label{lem:sp-lowest-frequency}
		For regular $\Omega$ and $\varrho\to\infty$,
		\begin{equation*}
			q_1^{\mathrm{free}}(0,\varrho\Omega)
			=C(\Omega)\varrho^{-\delta_m}
			\exp\!\left(-\pi\varrho
			\sum_{j=1}^{\kappa}j\sigma_j(\Omega)\right)
			\left(1+O(e^{-\eta(\Omega)\varrho})\right),
		\end{equation*}
		where $C(\Omega)>0$, $\eta(\Omega)>0$, and
		\[
		\delta_{2\kappa}=\kappa(\kappa-1),
		\qquad
		\delta_{2\kappa+1}=\kappa^2.
		\]
	\end{lemma}
	
	\begin{proof}
		For $x>0$ and $r\ge0$,
		\begin{equation*}
			\widehat w^{(r)}(x)
			=4\pi^2\sum_{\ell\ge1}
			(-1)^{\ell-1}\ell(-2\pi\ell)^r e^{-2\pi\ell x},
		\end{equation*}
		with absolute and locally uniform convergence. Insert this expansion into \Cref{lem:sp-free-radial} at $\varrho\Omega$. Expanding the determinant by columns gives a normally convergent series indexed by $\boldsymbol\ell=(\ell_1,\ldots,\ell_\kappa)\in\mathbb N^\kappa$, with exponential factor
		\[
		\exp\!\left(-\pi\varrho
		\sum_{k=1}^{\kappa}\ell_k\sigma_k\right).
		\]
		In even dimension the coefficient of this exponential in the derivative determinant is
		\[
		(4\pi^2)^\kappa(2\pi)^{\kappa(\kappa-1)}
		(-1)^{\sum_k(\ell_k-1)}
		\det[\ell_k^{2j-1}]_{j,k=1}^{\kappa},
		\]
		whereas in odd dimension it is
		\[
		(4\pi^2)^\kappa(-2\pi)^{\kappa^2}
		(-1)^{\sum_k(\ell_k-1)}
		\det[\ell_k^{2j}]_{j,k=1}^{\kappa}.
		\]
		The two determinants factor as
		\begin{align*}
			\det[\ell_k^{2j-1}]_{j,k=1}^{\kappa}
			&=\left(\prod_{k=1}^{\kappa}\ell_k\right)
			\prod_{j<k}(\ell_k^2-\ell_j^2),\\
			\det[\ell_k^{2j}]_{j,k=1}^{\kappa}
			&=\left(\prod_{k=1}^{\kappa}\ell_k^2\right)
			\prod_{j<k}(\ell_k^2-\ell_j^2).
		\end{align*}
		Thus a term vanishes exactly when two frequencies $\ell_j$ and $\ell_k$ coincide. For a fixed set $1\le r_1<\cdots<r_\kappa$, the strict rearrangement inequality, applied to $\sigma_1>\cdots>\sigma_\kappa>0$, assigns $r_k$ to $\sigma_k$ in the unique minimizing order. Since $r_k\ge k$, the unique minimum over all distinct positive frequencies is $\boldsymbol\ell=(1,\ldots,\kappa)$. Both displayed Vandermonde factors are nonzero at this vector.
		
		Put
		\[
		d_\Omega:=
		\min_{\substack{\boldsymbol\ell\in\mathbb N^\kappa,
				\,\ell_i\ne\ell_j\ (i\ne j)\\
				\boldsymbol\ell\ne(1,\ldots,\kappa)}}
		\left(\sum_{k=1}^{\kappa}\ell_k\sigma_k-
		\sum_{k=1}^{\kappa}k\sigma_k\right).
		\]
		This minimum exists and is positive. Indeed, $\sigma_\kappa>0$ makes the set of integer vectors below any fixed exponent level finite, and the preceding argument gives a unique minimizer. Choose $0<\eta(\Omega)<\pi d_\Omega/2$. The coefficients above have polynomial growth in $\boldsymbol\ell$, and, for $\varrho\ge1$, the remaining exponential factor is bounded by
		\[
		e^{-\eta(\Omega)\varrho}
		\exp\!\left[-\frac{\pi}{2}
		\left(\sum_k\ell_k\sigma_k-\sum_k k\sigma_k\right)\right].
		\]
		The resulting polynomially weighted series converges. Hence all remaining terms are bounded by the leading exponential times $O(e^{-\eta(\Omega)\varrho})$.
		
		The denominator in \eqref{eq:sp-free-radial-even} scales as $\varrho^{\kappa(\kappa-1)}$; in odd dimension the additional factor $\prod_j\sigma_j$ contributes $\varrho^\kappa$, giving $\delta_{2\kappa+1}=\kappa^2$. Hence the leading coefficient is nonzero. The heat kernel is positive, so that coefficient is positive.
	\end{proof}
	
	\begin{theorem}\label{thm:sp-free-distance}
		For every $\Omega\in\Lambda^2\pLie$,
		\begin{equation}\label{eq:sp-free-heat}
			-\lim_{t\downarrow0}4t\log q_t^{\mathrm{free}}(0,\Omega)
			=\Psi_{\pLie}(\Omega).
		\end{equation}
		Equivalently,
		\begin{equation}\label{eq:sp-free-distance}
			d_{\mathrm{free}}(e,(0,\Omega))^2
			=\Psi_{\pLie}(\Omega).
		\end{equation}
	\end{theorem}
	
	\begin{proof}
		The assertion is immediate for $m\le1$. For regular $\Omega$, the $2$-homogeneity of $\Delta_{\mathrm{free}}$ under the standard anisotropic dilations $(X,\Omega)\mapsto(rX,r^2\Omega)$ gives
		\begin{equation}\label{eq:sp-free-scaling}
			q_t^{\mathrm{free}}(0,\Omega)
			=t^{-m^2/2}q_1^{\mathrm{free}}(0,\Omega/t).
		\end{equation}
		Combining \eqref{eq:sp-free-scaling} with \Cref{lem:sp-lowest-frequency} yields
		\[
		\log q_t^{\mathrm{free}}(0,\Omega)
		=-\frac{\pi}{t}\sum_{j=1}^{\kappa}j\sigma_j(\Omega)
		+O_\Omega(\log(1/t)),
		\]
		and hence \eqref{eq:sp-free-heat} on the regular set.
		
		The free group is complete, the Lebesgue measure in exponential coordinates is Haar measure, and the left-invariant horizontal fields have zero divergence \cite{AgrachevBarilariBoscain,AgrachevBoscainGauthierRossi2009}. L\'eandre's logarithmic formula \cite[Theorem~17]{BarilariBoscainNeel2012} therefore identifies the regular rate with $d_{\mathrm{free}}^2$, proving \eqref{eq:sp-free-distance} for regular $\Omega$. The regular set is dense, while both the distance squared and $\Psi_{\pLie}$ are continuous. Thus \eqref{eq:sp-free-distance} extends to every $\Omega$. Applying the same logarithmic formula at the remaining points then gives \eqref{eq:sp-free-heat} everywhere.
	\end{proof}
	
	\begin{remark}
		The poles used in the residue calculation leading to \eqref{eq:sp-scalar-transform} are those of
		\[
		w(a)=\frac{a/2}{\sinh(a/2)},
		\]
		namely $2\pi i\Z\setminus\{0\}$. They produce the exponential levels $e^{-2\pi\ell x}$, $\ell\ge1$, in the expansion of $\widehat w(x)$. Orthogonal orbital integration produces an alternating determinant, so its radial columns must use distinct positive levels. The least possible levels are $1,\ldots,\kappa$. The relevant rank is the skew rank of a lift in $\Lambda^2\pLie$; the number of nonzero terms is half that rank.
	\end{remark}
	
	For an arbitrary two-step quotient, the corresponding vertical squared distance is obtained by minimizing $\Psi_{\pLie}$ over the affine fibers of the bracket map, as in the construction below. The free result descends to the tangent group and its loop problem. For the present symmetric pair, the tangent group $\mathbb N_{G,K}$ fixed in \eqref{eq:tangent-group-product} has product
	\begin{equation*}
		(X,Z)(Y,W)
		=
		\left(X+Y,Z+W+\frac12[X,Y]\right).
	\end{equation*}
	
	\begin{proposition}\label{prop:sp-central-quotient}
		The map
		\begin{equation}\label{eq:sp-quotient-map}
			q_B:\mathbb F_2(\pLie)\longrightarrow\mathbb N_{G,K},
			\qquad
			q_B(X,\Omega)=(X,B\Omega),
		\end{equation}
		is a surjective Carnot group homomorphism with central kernel $\ker B$. Its differential is the identity on the horizontal layer, so it is a submetry. Equip all linear spaces with the Lebesgue measures induced by their fixed inner products, and let $\mathrm d\sigma_Z$ be the induced Hausdorff measure on the affine fiber $B^{-1}(Z)$. Then
		\begin{equation}\label{eq:sp-fiber-heat}
			p_t^0(X,Z)
			=
			\det(BB^*)^{-1/2}
			\int_{B^{-1}(Z)}q_t^{\mathrm{free}}(X,\Omega)\,
			\mathrm d\sigma_Z(\Omega).
		\end{equation}
	\end{proposition}
	
	\begin{proof}
		The homomorphism identity follows from $B(X\wedge Y)=[X,Y]$, and surjectivity follows from \eqref{eq:sp-bracket-surjective}. Its kernel is $\{(0,\Omega):\Omega\in\ker B\}$ and therefore lies in the central second layer. If $u$ is a horizontal control, the free and quotient curves driven by $u$ have the same first-layer component, and their second-layer components are related by $B$. Thus every horizontal curve projects isometrically and every quotient curve has a horizontal lift with the same control and the same length. Taking infima of their lengths gives
		\begin{equation}\label{eq:sp-submetry}
			d_0(e,g)
			=
			\inf_{q_B(\widetilde g)=g}
			d_{\mathrm{free}}(e,\widetilde g).
		\end{equation}
		
		The sub-Laplacians intertwine under $q_B$, because the differential of $q_B$ is the identity on $\pLie$. Hence the pushforward of the free heat measure solves the heat equation on $\mathbb N_{G,K}$ with initial mass at the identity and must equal the tangent heat measure. For a test function $\varphi$, the linear map $B$ has constant normal Jacobian $J_B=\det(BB^*)^{1/2}$, and the linear coarea formula gives
		\begin{align*}
			&\int_{\pLie}\int_{\Lambda^2\pLie}
			\varphi(X,B\Omega)q_t^{\mathrm{free}}(X,\Omega)
			\,\mathrm d\Omega\,\mathrm dX \\
			&\quad=\det(BB^*)^{-1/2}
			\int_{\pLie}\int_{\kLie}\varphi(X,Z)
			\left(\int_{B^{-1}(Z)}q_t^{\mathrm{free}}(X,\Omega)
			\,\mathrm d\sigma_Z(\Omega)\right)\mathrm dZ\,\mathrm dX.
		\end{align*}
		Comparison with the tangent heat density proves \eqref{eq:sp-fiber-heat}.
	\end{proof}
	
	For $A\in\kLie$, let $\omega_1(A),\ldots,\omega_\kappa(A)$ be the nonnegative singular values of
	\[
	J_{B^*A}=\operatorname{ad}(A)|_{\pLie},
	\]
	as in \eqref{eq:sp-isotropy-bracket-bridge}. Take the Euclidean Fourier transform in the central variable with the convention
	\[
	\widehat f(A)=\int_{\kLie}e^{i\langle A,Z\rangle}f(Z)\,\mathrm dZ.
	\]
	The class-two Fourier-Mehler formula gives
	\[
	\widehat{p_t^0}(0,A)
	=(4\pi t)^{-m/2}
	\prod_{j=1}^{\kappa}w\bigl(2t\omega_j(A)\bigr);
	\]
	see \cite[formula~(5.5)]{Cygan1979}. Fourier inversion therefore yields
	\begin{equation}\label{eq:sp-tangent-fourier}
		p_t^0(0,Z)
		=
		\frac1{(2\pi)^r(4\pi t)^{m/2}}
		\int_{\kLie}e^{-i\langle A,Z\rangle}
		\prod_{j=1}^{\kappa}
		w\bigl(2t\omega_j(A)\bigr)\,\mathrm dA.
	\end{equation}
	Each zero factor is interpreted as $w(0)=1$ by continuity. Since $B$ is surjective by \eqref{eq:sp-bracket-surjective}, its adjoint $B^*$ is injective. Finite-dimensional norm equivalence gives $\sum_j\omega_j(A)\ge c|A|$; hence the product has exponential decay and the integral is absolutely convergent.
	
	\begin{remark}
		On a maximal torus of $K$, the nonzero $\omega_j(A)$ are the absolute values of the weights of the complexified isotropy representation, with multiplicity. Thus \eqref{eq:sp-tangent-fourier} may be written as a product over the positive isotropy weights.
	\end{remark}
	
	\begin{corollary}\label{thm:sp-tangent-distance}
		For every $Z\in\kLie$,
		\begin{equation*}
			-\lim_{t\downarrow0}4t\log p_t^0(0,Z)
			=
			d_0(e,(0,Z))^2
			=
			\mathfrak F_{G,K}(Z).
		\end{equation*}
		The fiber formula also satisfies
		\begin{equation*}
			-\lim_{t\downarrow0}4t\log
			\int_{B^{-1}(Z)}q_t^{\mathrm{free}}(0,\Omega)\,
			\mathrm d\sigma_Z(\Omega)
			=
			\mathfrak F_{G,K}(Z).
		\end{equation*}
	\end{corollary}
	
	\begin{proof}
		The lifts of $(0,Z)$ under \eqref{eq:sp-quotient-map} are precisely the $(0,\Omega)$ with $B\Omega=Z$. Hence \eqref{eq:sp-submetry}, \Cref{thm:sp-free-distance}, and \eqref{eq:sp-F} give
		\[
		d_0(e,(0,Z))^2
		=
		\min_{B\Omega=Z}
		d_{\mathrm{free}}(e,(0,\Omega))^2
		=
		\mathfrak F_{G,K}(Z).
		\]
		The group $\mathbb N_{G,K}$ is complete, the Lebesgue measure in exponential coordinates is Haar measure, and its left-invariant horizontal fields have zero divergence \cite{AgrachevBarilariBoscain,AgrachevBoscainGauthierRossi2009}. The logarithmic formula \cite[Theorem~17]{BarilariBoscainNeel2012}, followed by \eqref{eq:sp-fiber-heat}, give the heat kernel and fiber rates. Multiplication by $-4t$ removes the fixed coarea factor in the logarithmic limit.
	\end{proof}
	
	For a based path $x\in H^1([0,1];\pLie)$ with $x(0)=x(1)=0$, set
	\begin{equation*}
		\mathcal A_{G,K}(x)
		:=
		\frac12\int_0^1[x(s),\dot x(s)]\,\mathrm ds,
		\qquad
		\mathcal E(x):=\int_0^1|\dot x(s)|^2\,\mathrm ds.
	\end{equation*}
	
	\begin{corollary}\label{cor:sp-loop-isoperimetry}
		For every $Z\in\kLie$,
		\begin{equation}\label{eq:sp-loop-isoperimetry}
			\inf\left\{
			\mathcal E(x):
			x(0)=x(1)=0,\ \mathcal A_{G,K}(x)=Z
			\right\}
			=
			\mathfrak F_{G,K}(Z).
		\end{equation}
		The infimum is attained by a trigonometric polynomial using at most $\kappa$ positive frequencies.
	\end{corollary}
	
	\begin{proof}
		A horizontal curve in $\mathbb N_{G,K}$ ending at $(0,Z)$ has a based first-layer path $x$, central endpoint $\mathcal A_{G,K}(x)$, and energy $\mathcal E(x)$. Apply \Cref{thm:sp-tangent-distance}.
		
		For an explicit minimizer, choose $\Omega_Z\in B^{-1}(Z)$ minimizing \eqref{eq:sp-F} and write $\Omega_Z=\sum_j\sigma_j u_j\wedge v_j$ as in \eqref{eq:sp-skew-normal}. Then
		\begin{equation}\label{eq:sp-minimizing-loop}
			x_Z(s)
			:=
			\sum_{\sigma_j>0}
			\sqrt{\frac{\sigma_j}{\pi j}}
			\left(
			(\cos(2\pi js)-1)u_j+\sin(2\pi js)v_j
			\right)
		\end{equation}
		satisfies $x_Z(0)=x_Z(1)=0$ and
		\begin{equation*}
			\frac12\int_0^1x_Z\wedge\dot x_Z\,\mathrm ds=\Omega_Z,
			\qquad
			\mathcal E(x_Z)
			=4\pi\sum_j j\,\sigma_j
			=\mathfrak F_{G,K}(Z).
		\end{equation*}
		Applying $B$ gives the required central endpoint.
	\end{proof}
	
	It remains to compare the tangent horizontal loop with its endpoint in the compact symmetric pair. The passage from the tangent group back to $G$ uses the parity supplied by the symmetric decomposition. Odd Lie polynomials in horizontal variables belong to $\pLie$, while even ones belong to $\kLie$.
	
	\begin{lemma}
		\label{lem:sp-second-order-log}
		For every $M>0$ there are $\eta_M,C_M>0$ with the following property. Let $v\in L^2([0,1];\pLie)$ satisfy $\|v\|_{L^2}\le M$, and let $\gamma_\eta:[0,1]\to G$ be the absolutely continuous curve solving
		\[
		\gamma_\eta^{-1}\dot\gamma_\eta=\eta v,
		\qquad
		\gamma_\eta(0)=e.
		\]
		If $|\eta|\le\eta_M$, then $\gamma_\eta(1)$ lies in a fixed exponential neighborhood of $e$, and
		\begin{equation}\label{eq:sp-log-expansion}
			\log\gamma_\eta(1)
			=\eta X_v+\eta^2A_v+R_\eta(v),
		\end{equation}
		where
		\begin{equation*}
			X_v:=\int_0^1v(s)\,\mathrm ds,
			\qquad
			A_v:=\frac12\int_0^1[x_v(s),v(s)]\,\mathrm ds,
			\qquad
			x_v(s):=\int_0^s v(u)\,\mathrm du,
		\end{equation*}
		and
		\begin{equation}\label{eq:sp-parity-remainder}
			\left|\pr_{\pLie}R_\eta(v)\right|\le C_M|\eta|^3,
			\qquad
			\left|\pr_{\kLie}R_\eta(v)\right|\le C_M|\eta|^4.
		\end{equation}
	\end{lemma}
	
	\begin{proof}
		Embed $G$ faithfully in a unitary group $\U(N)$. For a complex parameter $z$, let $U_z:[0,1]\to\operatorname{GL}(N,\C)$ solve
		\[
		\dot U_z=U_z zv,
		\qquad U_z(0)=I.
		\]
		For real $z=\eta$, $U_\eta$ is the matrix realization of $\gamma_\eta$; thus $U_z$ is its complex-analytic extension in the parameter. The Dyson series is entire in $z$ and satisfies
		\[
		\|U_z(1)-I\|\le e^{|z|\|v\|_{L^1}}-1
		\le e^{|z|M}-1.
		\]
		Choose $r_M>0$ so that the right-hand side is smaller than $1/2$ for $|z|\le r_M$. The principal matrix logarithm is then holomorphic on this disc and uniformly bounded there. Hence
		\[
		\log U_z(1)=\sum_{r\ge1}z^r\mathcal M_r(v),
		\]
		and Cauchy's estimate gives
		\begin{equation}\label{eq:sp-magnus-coefficient-bound}
			|\mathcal M_r(v)|\le C_M r_M^{-r},
			\qquad r\ge1,
		\end{equation}
		uniformly for $\|v\|_{L^2}\le M$. The first two coefficients are the standard Magnus terms
		\[
		\mathcal M_1(v)=X_v,
		\qquad
		\mathcal M_2(v)=\frac12\int_0^1\int_0^s[v(u),v(s)]\,\mathrm du\,\mathrm ds
		=A_v.
		\]
		By \eqref{eq:sp-brackets}, every homogeneous Lie polynomial of odd degree in horizontal variables lies in $\pLie$, while every one of even degree lies in $\kLie$. Taking $\eta_M<r_M/2$ and summing \eqref{eq:sp-magnus-coefficient-bound} over odd degrees from three and even degrees from four proves \eqref{eq:sp-parity-remainder}.
	\end{proof}
	
	Fix an exponential neighborhood $e\in U\subset G$ and let $\log:U\to\gLie$ be the local inverse of the exponential map. For $g\in U$, write $\log g=P+W$ with $P\in\pLie$ and $W\in\kLie$. We assign weights one and two to $P$ and $W$, respectively.
	
	\begin{lemma}
		\label{lem:sp-endpoint-correction}
		Fix $R>0$. Suppose that there is $\varepsilon_0>0$ such that $g_\varepsilon\in U$ for $0<\varepsilon<\varepsilon_0$ and that
		\[
		\log g_\varepsilon
		=\varepsilon Z+P_\varepsilon+W_\varepsilon,
		\qquad
		Z,W_\varepsilon\in\kLie,\quad P_\varepsilon\in\pLie,
		\]
		where $|Z|\le R$, $|P_\varepsilon|\le C_R\varepsilon^{3/2}$, and $|W_\varepsilon|\le C_R\varepsilon^2$. Then, for small $\varepsilon$,
		\[
		h_\varepsilon
		:=\exp(-\varepsilon Z)g_\varepsilon\exp(-P_\varepsilon)
		\]
		satisfies
		\[
		|\pr_{\pLie}\log h_\varepsilon|=O_R(\varepsilon^{5/2}),
		\qquad
		|\pr_{\kLie}\log h_\varepsilon|=O_R(\varepsilon^2).
		\]
		Consequently
		\begin{equation}\label{eq:sp-residual-distance}
			d_{\sr}(e,h_\varepsilon)=O_R(\varepsilon).
		\end{equation}
	\end{lemma}
	
	\begin{proof}
		Put $U=\varepsilon Z$, $P=P_\varepsilon$, and $W=W_\varepsilon$. The Baker-Campbell-Hausdorff (BCH) series converges uniformly in a fixed neighborhood of the origin. The first multiplication gives
		\begin{equation*}
			\log\!\left(e^{-U}e^{U+P+W}\right)
			=P+W-\frac12[U,P]+R_1,
		\end{equation*}
		where the bracket relations \eqref{eq:sp-brackets} and the assumed sizes imply
		\[
		|\pr_{\pLie}R_1|=O_R(\varepsilon^{7/2}),
		\qquad
		|\pr_{\kLie}R_1|=O_R(\varepsilon^3).
		\]
		Indeed, the omitted horizontal terms contain either two copies of $U$ and one of $P$, or one copy each of $P$ and $W$; the omitted vertical terms start with $[U,W]$.
		
		Applying BCH once more to the right multiplication by $e^{-P}$ gives
		\begin{equation*}
			\log h_\varepsilon
			=W-\frac12[U,P]+R_2,
		\end{equation*}
		with
		\[
		|\pr_{\pLie}R_2|=O_R(\varepsilon^{7/2}),
		\qquad
		|\pr_{\kLie}R_2|=O_R(\varepsilon^3).
		\]
		Thus the horizontal and vertical components of $\log h_\varepsilon$ are $O_R(\varepsilon^{5/2})$ and $O_R(\varepsilon^2)$, respectively. The local ball-box estimate in exponential coordinates adapted to the weight-one space $\pLie$ and weight-two space $\kLie$ gives
		\[
		d_{\sr}(e,h_\varepsilon)
		\le C_R\left(
		|\pr_{\pLie}\log h_\varepsilon|
		+|\pr_{\kLie}\log h_\varepsilon|^{1/2}
		\right),
		\]
		which proves \eqref{eq:sp-residual-distance}; see \cite[Theorem~10.67]{AgrachevBarilariBoscain}.
	\end{proof}
	
	\begin{theorem}
		\label{thm:sp-compact-vertical}
		For every $R>0$ there are $\varepsilon_R,C_R>0$ such that, whenever $\mathfrak F_{G,K}(Z)\le R$ and $0<\varepsilon\le\varepsilon_R$,
		\begin{equation}\label{eq:sp-compact-vertical}
			\left|
			d_{\sr}(e,\exp(\varepsilon Z))^2
			-\varepsilon\mathfrak F_{G,K}(Z)
			\right|
			\le C_R\varepsilon^{3/2}.
		\end{equation}
		The lower estimate has the stronger remainder
		\begin{equation}\label{eq:sp-compact-lower}
			d_{\sr}(e,\exp(\varepsilon Z))^2
			\ge
			\varepsilon\mathfrak F_{G,K}(Z)-C_R\varepsilon^2.
		\end{equation}
	\end{theorem}
	
	\begin{proof}
		By \Cref{prop:sp-regularity}, the sublevel $\{\mathfrak F_{G,K}\le R\}$ is compact. Choose a minimizer $\Omega_Z$ in \eqref{eq:sp-F}. By \eqref{eq:sp-coercive},
		\[
		|\Omega_Z|\le\frac{\Psi_{\pLie}(\Omega_Z)}{4\pi}
		=\frac{\mathfrak F_{G,K}(Z)}{4\pi}\le\frac{R}{4\pi}.
		\]
		The loop $x_Z$ in \eqref{eq:sp-minimizing-loop} has energy at most $R$. The constants below are therefore uniform on this sublevel.
		
		For the upper bound, define the horizontal control
		\[
		v_\varepsilon(s):=\sqrt\varepsilon\,\dot x_Z(s)\in\pLie,
		\]
		and let $\gamma_\varepsilon:[0,1]\to G$ solve $\gamma_\varepsilon^{-1}\dot\gamma_\varepsilon=v_\varepsilon$ with $\gamma_\varepsilon(0)=e$. Set $g_\varepsilon:=\gamma_\varepsilon(1)$. Since $x_Z$ is based, the quantities in \Cref{lem:sp-second-order-log} satisfy $X_{\dot x_Z}=0$ and $A_{\dot x_Z}=Z$. Hence that lemma gives
		\[
		\log g_\varepsilon
		=\varepsilon Z+P_\varepsilon+L_\varepsilon,
		\qquad
		|P_\varepsilon|\le C_R\varepsilon^{3/2},
		\quad
		|L_\varepsilon|\le C_R\varepsilon^2.
		\]
		Follow this curve by the horizontal segment $s\mapsto g_\varepsilon\exp(-sP_\varepsilon)$, whose length is $O_R(\varepsilon^{3/2})$ and whose endpoint is
		\[
		g_\varepsilon\exp(-P_\varepsilon)
		=\exp(\varepsilon Z)h_\varepsilon.
		\]
		By \Cref{lem:sp-endpoint-correction}, a horizontal curve from $e$ to $h_\varepsilon$ has length $O_R(\varepsilon)$. Left-translating its reversal corrects the last endpoint to $\exp(\varepsilon Z)$. The resulting curve has length at most
		\[
		\sqrt{\varepsilon\mathfrak F_{G,K}(Z)}+C_R\varepsilon.
		\]
		Squaring yields
		\begin{equation}\label{eq:sp-compact-upper}
			d_{\sr}(e,\exp(\varepsilon Z))^2
			\le
			\varepsilon\mathfrak F_{G,K}(Z)+C_R\varepsilon^{3/2}.
		\end{equation}
		
		For the lower bound, \eqref{eq:sp-compact-upper} first gives $d_{\sr}(e,\exp(\varepsilon Z))^2\le C_R\varepsilon$. Choose a horizontal curve $\gamma_\varepsilon$ from $e$ to $\exp(\varepsilon Z)$ whose squared length is at most $d_{\sr}(e,\exp(\varepsilon Z))^2+\varepsilon^2$. Reparametrize this curve at constant speed on $[0,1]$; its energy then equals its squared length. Write its control as
		\[
		\gamma_\varepsilon^{-1}\dot\gamma_\varepsilon
		=\sqrt\varepsilon\,v_\varepsilon,
		\qquad
		E_\varepsilon:=\int_0^1|v_\varepsilon|^2\,\mathrm ds.
		\]
		Then $E_\varepsilon\le C_R$. With $x_\varepsilon(s)=\int_0^s v_\varepsilon(u)\,\mathrm du$ and $X_\varepsilon=x_\varepsilon(1)$, projection of \eqref{eq:sp-log-expansion} onto $\pLie$ and $\kLie$ gives
		\begin{equation*}
			|X_\varepsilon|=O_R(\varepsilon),
			\qquad
			\mathcal A_{G,K}(x_\varepsilon)=Z+O_R(\varepsilon).
		\end{equation*}
		Close the first-layer path by subtracting its linear endpoint segment:
		\[
		\widetilde x_\varepsilon(s)
		:=x_\varepsilon(s)-sX_\varepsilon.
		\]
		Then
		\begin{equation*}
			\int_0^1|\dot{\widetilde x}_\varepsilon|^2\,\mathrm ds
			=E_\varepsilon-|X_\varepsilon|^2\le E_\varepsilon.
		\end{equation*}
		The change of area is given exactly by
		\begin{equation*}
			\mathcal A_{G,K}(\widetilde x_\varepsilon)
			-\mathcal A_{G,K}(x_\varepsilon)
			=\frac12\int_0^1(1-2s)[X_\varepsilon,v_\varepsilon(s)]\,\mathrm ds.
		\end{equation*}
		Since $|X_\varepsilon|=O_R(\varepsilon)$ and $\|v_\varepsilon\|_{L^2}=O_R(1)$, Cauchy-Schwarz and the continuity of the bracket give
		\begin{equation*}
			\mathcal A_{G,K}(\widetilde x_\varepsilon)
			=Z+O_R(\varepsilon).
		\end{equation*}
		Using \Cref{cor:sp-loop-isoperimetry,prop:sp-regularity},
		\[
		\mathfrak F_{G,K}(Z)
		\le
		\mathfrak F_{G,K}
		\bigl(\mathcal A_{G,K}(\widetilde x_\varepsilon)\bigr)
		+C_R\varepsilon
		\le E_\varepsilon+C_R\varepsilon.
		\]
		Multiplication by $\varepsilon$, the energy-length identity after constant-speed reparametrization, and the chosen almost minimality prove \eqref{eq:sp-compact-lower}. Combining \eqref{eq:sp-compact-lower} and \eqref{eq:sp-compact-upper} proves \eqref{eq:sp-compact-vertical}.
	\end{proof}
	
	The finite formulation \eqref{eq:sp-finite-objective}, the enclosure \eqref{eq:sp-grid-enclosure}, and \Cref{thm:sp-compact-vertical} prove the general assertion of \Cref{thm:intro-vertical}; its explicit one-column AIII formula is supplied by \Cref{thm:examples-AIII}.
	
	The heat kernel consequences of the compact comparison are recorded next. The logarithmic formula \eqref{eq:leandre-logarithmic} and \Cref{thm:sp-compact-vertical} give, uniformly for $\mathfrak F_{G,K}(Z)\le R$,
	\begin{equation}\label{eq:sp-compact-heat}
		-\lim_{t\downarrow0}4t\log p_t(\exp(\varepsilon Z))
		=
		\varepsilon\mathfrak F_{G,K}(Z)+O_R(\varepsilon^{3/2}),
		\qquad \varepsilon\downarrow0.
	\end{equation}
	The limits are iterated: first $t\downarrow0$ at fixed $\varepsilon$, then $\varepsilon\downarrow0$.
	
	The tangent kernel also gives the leading parabolic rescaling. The pointwise nilpotentization theorem of Colin de Verdi\`ere, Hillairet, and Tr\'elat \cite[Theorem~A]{ColinHillairetTrelat2021} gives
	\begin{equation*}
		\varepsilon^{Q/2}
		p_{\varepsilon t}(\exp(\varepsilon Z))
		\longrightarrow
		\vol(G)\,p_t^0(0,Z),
		\qquad t>0,
	\end{equation*}
	locally uniformly in $Z$, with $p_t^0$ normalized relative to Euclidean measure on $\pLie\oplus\kLie$. This fixed-time nilpotentization complements the iterated logarithmic limit in \eqref{eq:sp-compact-heat}.
	
	Substitution in the integral formula \Cref{thm:integral-formula} gives the equivalent identity
	\begin{align*}
		&-\lim_{t\downarrow0}4t\log\Biggl[
		e^{-|\rho_K|^2t}
		\int_{\kLie}q_t^{\C}
		\bigl(\exp(\varepsilon Z)e^{iY}\bigr)
		j_K^{\nc}(Y)^{1/2}
		\frac{e^{-|Y|^2/(4t)}}{(4\pi t)^{r/2}}\,\mathrm dY
		\Biggr]\\
		&\hspace{12em}
		=
		\varepsilon\mathfrak F_{G,K}(Z)+O_R(\varepsilon^{3/2})\quad  (\varepsilon \downarrow 0).
	\end{align*}
	
	\begin{remark}
		The free formula, central quotient, tangent heat kernel, and loop isoperimetry extend to an arbitrary two-step compact pair after replacing \eqref{eq:sp-B} by
		\[
		B(X\wedge Y)=\pr_{\lLie}[X,Y],
		\qquad B:\Lambda^2\pLie\to\lLie.
		\]
		The symmetric pair relations supply the odd-even separation in \Cref{lem:sp-second-order-log}, which yields the $O_R(\varepsilon^{3/2})$ compact remainder in \eqref{eq:sp-compact-vertical}.
	\end{remark}
	\section{Spectral zeta functions and regularized determinants}
	\label{zeta:section}
	
	Assume throughout this section that $(G,L)$ is two-step and suppose $\dim G>0$. We use the standing homogeneous dimension $Q=m+2r$. Spectral zeta functions of sub-Laplacians on compact two-step nilmanifolds were studied in \cite{BauerFurutaniIwasaki2012,BauerFurutaniIwasaki2015}. The two-step condition is bracket generating, so the kernel of $-\Delta_{\pLie}$ consists of the constants. Let
	\[
	0=\lambda_0<\lambda_1\le\lambda_2\le\cdots
	\]
	be its eigenvalues, repeated with scalar multiplicity. The reduced spectral zeta function is
	\begin{equation*}
		\zeta_{\pLie}(s)
		:=\Tr'\!\left((-\Delta_{\pLie})^{-s}\right)
		=\sum_{j\ge1}\lambda_j^{-s},
		\qquad \Re s>Q/2,
	\end{equation*}
	where the prime means that the zero mode is omitted.
	
	For every $t>0$, hypoellipticity gives $e^{t\Delta_{\pLie}}$ a smooth kernel on the compact group $G$, so the operator is trace class. Since its kernel is $K_t(g,h)=p_t(g^{-1}h)$ and Haar measure has mass one,
	\begin{equation}\label{zeta:heat-trace-identity}
		\Tr\!\left(e^{t\Delta_{\pLie}}\right)
		=\int_G K_t(g,g)\,\mathrm dg=p_t(e).
	\end{equation}
	
	We first analyze the heat trace and its local zeta data. For $Y\in\lLie$, let $J_Y:=\operatorname{ad}(Y)|_{\pLie}$. In the symmetric specialization $L=K$, this is the same operator as $J_{B^*Y}$ by \eqref{eq:sp-isotropy-bracket-bridge}. Reusing the Jacobians fixed in \Cref{sec:preliminaries}, define
	\begin{equation*}
		\mathcal D_{G,L}(Y)
		:=\left(\frac{j_L^{\nc}(Y)}{j_G^{\nc}(Y)}\right)^{1/2}
		=\left[\det\!\left(
		\frac{|J_Y|}{2\sinh(|J_Y|/2)}\right)\right]^{1/2},
	\end{equation*}
	where $|J_Y|=(-J_Y^2)^{1/2}$ and zero eigenvalues are interpreted by continuity. The two-step condition makes $Y\mapsto J_Y$ injective. Indeed, if $[Y,\pLie]=0$, metric invariance and $\gLie=\pLie+[\pLie,\pLie]$ give $Y=0$. Compactness of the unit sphere in $\lLie$ then gives $\Tr|J_Y|\ge c|Y|$. Hence $x/(2\sinh(x/2))\le C(1+x)e^{-x/2}$ implies
	\begin{equation}\label{zeta:relative-decay}
		0<\mathcal D_{G,L}(Y)
		\le C(1+|Y|)^M e^{-c|Y|}
	\end{equation}
	for suitable $c,C>0$ and $M\ge0$. In particular, $\mathcal D_{G,L}$ has finite moments of every order.
	
	Set
	\begin{equation}\label{zeta:local-constants}
		C_{G,L}:=
		\frac{\vol(G)}{(4\pi)^{Q/2}}
		\int_{\lLie}\mathcal D_{G,L}(Y)\,\mathrm dY,
		\qquad
		\beta_{G,L}:=|\rho_G|^2-|\rho_L|^2.
	\end{equation}
	
	These constants can be evaluated for an arbitrary compact connected two-step pair. This separates the metric Lie algebra data from the global normalization of the compact form of $G$.
	
	Let $T\subset L$ be the maximal torus fixed in \Cref{sec:preliminaries}, with positive roots $R_L^+$, Weyl group $\Wt_L$, and
	\[
	\Pi_L(H)=\prod_{\alpha\in R_L^+}\alpha(H).
	\]
	Write the nonzero weights of the real $T$-module $\pLie_{\C}$ as
	\[
	\operatorname{Wt}(\pLie_{\C})\setminus\{0\}
	=\coprod_{\omega\in\Sigma_{\pLie}^+}
	m_\omega\{\omega,-\omega\},
	\]
	where $\Sigma_{\pLie}^+$ contains one representative from each pair $\{\omega,-\omega\}$ and $m_\omega$ is its multiplicity.
	
	Enumerate the distinct positive isotropy weights and their multiplicities as
	\[
	\omega_1,\ldots,\omega_{N_{\mathrm{iso}}},
	\qquad
	m_1,\ldots,m_{N_{\mathrm{iso}}}
	\]
	respectively. Subdivide $\tLie$ by their kernels and the compact root hyperplanes. On a chamber $\mathcal C$, choose signs so that
	\[
	\ell_{\mathcal C,j}=\pm\omega_j,
	\qquad
	a_{\mathcal C,\alpha}=\pm\alpha
	\]
	are positive on $\mathcal C$, and put
	\begin{align*}
		P_{\mathcal C}(H)
		&:=\prod_{\alpha\in R_L^+}a_{\mathcal C,\alpha}(H)^2
		\prod_{j=1}^{N_{\mathrm{iso}}}\ell_{\mathcal C,j}(H)^{m_j},
		\\
		A_{\mathcal C}(\mathbf n)
		&:=\sum_{j=1}^{N_{\mathrm{iso}}}
		\left(n_j+\frac{m_j}{2}\right)\ell_{\mathcal C,j},
		\qquad \mathbf n\in\mathbb N_0^{N_{\mathrm{iso}}}.
	\end{align*}
	The isotropy weights span $\tLie^*$ because $Y\mapsto J_Y$ is injective. Thus each chamber closure $\overline{\mathcal C}$ is \emph{pointed}, meaning $\overline{\mathcal C}\cap(-\overline{\mathcal C})=\{0\}$, or equivalently that it contains no nonzero line. Choose a rational simplicial subdivision with ray generators $v_{\varsigma,1},\ldots,v_{\varsigma,r_T}$ in the cocharacter lattice, where $r_T=\dim T$, and expand
	\begin{equation*}
		P_{\mathcal C}\!\left(\sum_{i=1}^{r_T}t_iv_{\varsigma,i}\right)
		=\sum_\mu p_{\mathcal C,\varsigma,\mu}t^\mu
	\end{equation*}
	over its nonzero monomials.
	
	The following lemma converts the Cartan integral into elementary Laplace transforms on simplicial cones. Pointedness makes every linear denominator strictly positive, while positivity permits Tonelli's theorem. This is the standard calculation of a Laplace transform on a simplicial cone; compare \cite[\S1.5]{BrionVergne1997}.
	
	\begin{lemma}\label{lem:polyhedral-chambers}
		The hyperplanes of the compact roots and isotropy weights form a finite rational arrangement in $\tLie$. Each chamber closure is a pointed rational polyhedral cone and admits a subdivision into rational simplicial cones whose ray generators can be chosen in the cocharacter lattice. On every such cone, the polynomial $P_{\mathcal C}$ has nonnegative coefficients in the ray coordinates, and
		\[
		A_{\mathcal C}(\mathbf n)(v_{\varsigma,i})>0
		\qquad(\mathbf n\in\mathbb N_0^{N_{\mathrm{iso}}}).
		\]
		The overlaps between distinct cones have Lebesgue measure zero.
	\end{lemma}
	
	\begin{proof}
		Roots and isotropy weights are integral on the cocharacter lattice, so the arrangement is rational. If a chamber closure contained a line, every isotropy weight would vanish on that line; the weights span $\tLie^*$, so the line is zero. Standard polyhedral subdivision gives rational simplicial cones, and scaling their rays gives cocharacter lattice generators. Every signed linear form used in $P_{\mathcal C}$ is nonnegative on the chamber; its coefficients in the ray coordinates are therefore nonnegative, and so are the coefficients of their product. On a nonzero ray, at least one isotropy weight is positive. Since each coefficient $n_j+m_j/2$ is positive, the displayed evaluation of $A_{\mathcal C}(\mathbf n)$ is positive. Boundaries of distinct cones lie in proper linear subspaces and have measure zero.
	\end{proof}
	
	\begin{theorem}\label{zeta:complete-local-constants}
		For every compact connected two-step pair $(G,L)$,
		\begin{equation}\label{zeta:curvature-defect}
			\beta_{G,L}
			=\frac{\operatorname{Scal}_G-\operatorname{Scal}_L}{6}
			=\frac1{12}\sum_{a,i}|[Y_a,X_i]|^2
			+\frac1{24}\sum_{i,j}|[X_i,X_j]|^2.
		\end{equation}
		In particular, $\beta_{G,L}\ge0$. For the tangent group $\mathbb N_{G,L}$ and its heat kernel $p_t^0$,
		\begin{equation}\label{zeta:nilpotent-model}
			p_t^0(e)
			=\frac{t^{-Q/2}}{(4\pi)^{Q/2}}
			\int_{\lLie}\mathcal D_{G,L}(Y)\,\mathrm dY,
			\qquad
			C_{G,L}=\vol(G)p_1^0(e).
		\end{equation}
		Equivalently,
		\begin{align}
			C_L&=\frac{\vol(L)}{\vol(T)|\Wt_L|},
			\label{zeta:weyl-constant}\\
			\mathcal D_{G,L}(H)
			&=\prod_{\omega\in\Sigma_{\pLie}^+}
			w\bigl(\omega(H)\bigr)^{m_\omega},
			\qquad H\in\tLie,
			\label{zeta:weight-product}
		\end{align}
		and
		\begin{equation}\label{zeta:root-data-constant}
			C_{G,L}
			=\frac{\vol(G)\vol(L)}
			{(4\pi)^{Q/2}\vol(T)|\Wt_L|}
			\int_{\tLie}\Pi_L(H)^2
			\prod_{\omega\in\Sigma_{\pLie}^+}
			w\bigl(\omega(H)\bigr)^{m_\omega}\,\mathrm dH.
		\end{equation}
		The remaining Cartan integral has the integral-free expansion
		\begin{align}
			\int_{\tLie}\Pi_L(H)^2
			\prod_{j=1}^{N_{\mathrm{iso}}}
			w\bigl(\omega_j(H)\bigr)^{m_j}\,\mathrm dH
			&=\sum_{\mathcal C,\varsigma,\mu}
			\operatorname{vol}_{\tLie}
			(v_{\varsigma,1},\ldots,v_{\varsigma,r_T})
			p_{\mathcal C,\varsigma,\mu}
			\prod_{i=1}^{r_T}\mu_i!
			\notag\\
			&\qquad\quad\times
			\sum_{\mathbf n\in\mathbb N_0^{N_{\mathrm{iso}}}}
			\frac{\displaystyle
				\prod_{j=1}^{N_{\mathrm{iso}}}
				\binom{n_j+m_j-1}{m_j-1}}
			{\displaystyle
				\prod_{i=1}^{r_T}
				\bigl(A_{\mathcal C}(\mathbf n)
				(v_{\varsigma,i})\bigr)^{\mu_i+1}}.
			\label{zeta:polyhedral-dirichlet-formula}
		\end{align}
		All terms are nonnegative, and the multiple Dirichlet series converges absolutely.
	\end{theorem}
	
	\begin{proof}
		For a compact connected group $H$ with a bi-invariant metric, apply the ordinary image formula to its heat kernel at the identity. The zero lattice vector gives
		\[
		q_t^H(e)=\frac{\vol(H)}{(4\pi t)^{\dim H/2}}
		e^{t|\rho_H|^2}+O(t^{-N}e^{-c/t}).
		\]
		Comparison with the universal first heat coefficient \cite[Chapter~1]{Gilkey}, $\vol(H)(4\pi t)^{-\dim H/2}(1+\operatorname{Scal}_H t/6+O(t^2))$, gives $\operatorname{Scal}_H=6|\rho_H|^2$ in the present normalization. Now expand $\operatorname{Scal}_G=\frac14\sum_{u,v}|[E_u,E_v]|^2$ in an orthonormal basis adapted to $\gLie=\lLie\oplus\pLie$ and subtract the corresponding formula for $L$. Dividing by six gives \eqref{zeta:curvature-defect}.
		
		The general class-two heat kernel formula gives \eqref{zeta:nilpotent-model} \cite{Cygan1979}. On $\tLie$, the nonzero eigenvalues of $|J_H|$ occur as $|\omega(H)|$ with multiplicity $m_\omega$, which proves \eqref{zeta:weight-product}. The metric measure Weyl formula is
		\[
		\int_{\lLie}f(Y)\,\mathrm dY
		=\frac{\vol(L)}{\vol(T)|\Wt_L|}
		\int_{\tLie}f(H)\Pi_L(H)^2\,\mathrm dH
		\]
		for every integrable $\Ad(L)$-invariant function $f$. Applying it to $\mathcal D_{G,L}$ proves \eqref{zeta:weyl-constant} and \eqref{zeta:root-data-constant}.
		
		By \Cref{lem:polyhedral-chambers}, the chamber decomposition and its simplicial subdivisions have the stated positivity properties. On a chamber $\mathcal C$,
		\[
		w\bigl(\omega_j(H)\bigr)^{m_j}
		=\ell_{\mathcal C,j}(H)^{m_j}
		e^{-m_j\ell_{\mathcal C,j}(H)/2}
		\bigl(1-e^{-\ell_{\mathcal C,j}(H)}\bigr)^{-m_j}.
		\]
		Expand the last factor as
		\[
		(1-e^{-\ell})^{-m}
		=\sum_{n\ge0}\binom{n+m-1}{m-1}e^{-n\ell}.
		\]
		Tonelli's theorem applies because every term is nonnegative. On a simplicial cone $H=\sum_i t_i v_{\varsigma,i}$,
		\[
		\int_0^\infty t_i^{\mu_i}
		e^{-A_{\mathcal C}(\mathbf n)(v_{\varsigma,i})t_i}\,\mathrm dt_i
		=\frac{\mu_i!}
		{A_{\mathcal C}(\mathbf n)(v_{\varsigma,i})^{\mu_i+1}}.
		\]
		Multiplication gives \eqref{zeta:polyhedral-dirichlet-formula}. Pointedness and the interior-dual position of $A_{\mathcal C}(\mathbf n)$ make every denominator positive; finiteness of the original integral gives absolute convergence. The metric volumes in the prefactor are computable from the metric cocharacter data by Macdonald's formula \cite[pp.~660-661]{Hashimoto}.
	\end{proof}
	
	The behavior under central factors separates local from global data. A connected central torus is part of the metric Lie algebra pair and may change $Q$, the isotropy weights, and the volumes. A finite central local isometry quotient leaves $\mathcal D_{G,L}$, $p_1^0(e)$, and $\beta_{G,L}$ unchanged. More precisely, if $\pi:\widetilde G\to G$ has kernel $F$ and $\widetilde L^{\circ}$ is the identity component of $\pi^{-1}(L)$, then
	\begin{equation*}
		C_{\widetilde G,\widetilde L^{\circ}}=|F|C_{G,L},
		\qquad
		\beta_{\widetilde G,\widetilde L^{\circ}}=\beta_{G,L}.
	\end{equation*}
	The affine lattice, lattice sign, and character sector restrictions are carried by the global remainder, while the local constants depend on the metric Lie algebra pair.
	
	The following fact will be used in both the local and global estimates. Let $U\subset\tLie_{G,\C}$ be a Weyl-invariant open set and let $F$ be a holomorphic Weyl anti-invariant function on $U$, meaning $F(wH)=\det(w)F(H)$ for $w\in\Wt_G$. Then $F/\Pi_G$ extends holomorphically across the root hyperplanes. If $S$ is a compact subset of $U$, its supremum on $S$ is bounded by derivatives of $F$ of order at most $|R_G^+|$ on a slightly larger compact subset of $U$. In particular, if $U=\tLie_{G,\C}$ and $H\in\tLie_G$, then
	\begin{equation}\label{zeta:weyl-division-bound}
		\left|\frac{F(H)}{\Pi_G(H)}\right|
		\le C_G\max_{0\le j\le |R_G^+|}
		\sup_{|X|\le |H|}\|\nabla^jF(X)\|.
	\end{equation}
	
	Indeed, anti-invariance makes $F$ vanish on the fixed hyperplane of every root reflection. For $\alpha\in R_G^+$, put
	\[
	n_\alpha:=\frac{\alpha^\sharp}{|\alpha|^2},
	\qquad
	\pi_\alpha H:=H-\alpha(H)n_\alpha.
	\]
	Thus $\alpha(n_\alpha)=1$ and $\pi_\alpha$ is orthogonal projection onto $\ker\alpha$. The fundamental theorem of calculus gives the explicit division formula
	\begin{equation}\label{zeta:single-root-division}
		\frac{F(H)}{\alpha(H)}
		=\int_0^1
		(\partial_{n_\alpha}F)
		\bigl(\pi_\alpha H+u\alpha(H)n_\alpha\bigr)\,\mathrm du,
	\end{equation}
	first off $\ker\alpha$ and then everywhere by holomorphic continuation. Hence $F=\alpha F_\alpha$ with $F_\alpha$ holomorphic. If $\beta\ne\alpha$, then $F_\alpha=F/\alpha$ vanishes on the open dense part of $\ker\beta$ on which $\alpha\ne0$, and therefore on all of $\ker\beta$. Repeating the division over the positive roots gives $F=\Pi_G\widetilde F$ with $\widetilde F$ holomorphic on $U$.
	
	Iterating \eqref{zeta:single-root-division} expresses $\widetilde F(H)$ as a finite integral of derivatives of $F$ of order at most $|R_G^+|$. Orthogonal projection and the intervening line segments do not increase the norm, so every evaluation point for real $H$ lies in the ball $\{|X|\le|H|\}$. This proves \eqref{zeta:weyl-division-bound}. For a compact subset of $U$, choose a slightly larger compact neighborhood containing all the finitely many projected segments; the same integral formulas and Cauchy estimates give the asserted complex Cartan bound.
	
	\begin{proposition}
		For every compact connected $G$ there exist integers $A,M,N\ge0$ and constants $c,C>0$ such that, for $0<t\le1$ and $Y\in\gLie$,
		\begin{align}
			&j_G^{\nc}(Y)^{1/2}e^{-|Y|^2/(4t)}q_t^{\C}(e^{iY})
			=\vol(G)\frac{e^{t|\rho_G|^2}}{(4\pi t)^{\dim G/2}}
			+R_t(Y),
			\label{zeta:ambient-small-time-estimate}
		\end{align}
		where
		\begin{equation}\label{zeta:ambient-small-time-bound}
			|R_t(Y)|\le Ct^{-N}e^{-c/t}(1+|Y|)^M.
		\end{equation}
		More precisely, each nonzero Weyl orbit $\mathcal O$ in the semisimple lattice satisfies
		\begin{equation}\label{zeta:ambient-orbit-bound}
			\left|
			\frac{\sum_{\gamma\in\mathcal O}\varepsilon_G(\gamma)
				\Pi_G(Y-i\gamma)e^{-i\langle\gamma,Y\rangle/(2t)}}
			{\Pi_G(Y)}
			\right|
			\le Ct^{-N}(1+|Y|+\ell_{\mathcal O})^A.
		\end{equation}
	\end{proposition}
	
	\begin{proof}
		Conjugation invariance reduces the estimate to a maximal torus. Specialize \Cref{thm:heat-kernel-image} to $Z=iY$ and multiply by $j_G^{\nc}(Y)^{1/2}e^{-|Y|^2/(4t)}$. The zero lattice vector gives the first term in \eqref{zeta:ambient-small-time-estimate}. For a nonzero semisimple orbit of the Weyl group, the numerator in \eqref{zeta:ambient-orbit-bound} is Weyl anti-invariant. Iterating \eqref{zeta:single-root-division} shows that each derivative either acts on $\Pi_G(Y-i\gamma)$ or on the phase, producing at most a factor $C(1+|Y|+|\gamma|)$ or $C|\gamma|/t$. Since a Weyl orbit has uniformly bounded cardinality, \eqref{zeta:weyl-division-bound} proves \eqref{zeta:ambient-orbit-bound}.
		
		Write the common toral Lie algebra of the finite cover used in \Cref{thm:heat-kernel-image} as
		\[
		\widetilde{\tLie}_G=\mathfrak z(\gLie)\oplus\tLie_{G,\mathrm{ss}},
		\qquad
		\widetilde\Lambda=\Lambda_{\mathfrak z}\oplus\Lambda_{\mathrm{ss}}.
		\]
		By \eqref{eq:central-lattice-sequence}, the actual lattice is a finite union
		\[
		\Lambda_G=\bigcup_{f\in F}(\eta_f+\widetilde\Lambda)
		\]
		of product lattice cosets. The Weyl group acts trivially on the central factor. A pure central vector therefore contributes only its Gaussian and unit-modulus phase. A pure semisimple orbit satisfies \eqref{zeta:ambient-orbit-bound}. For a mixed vector, its central component is fixed along the Weyl orbit, while the semisimple component is governed by the iterated Weyl division estimate \eqref{zeta:weyl-division-bound}, obtained by repeating the one-root formula \eqref{zeta:single-root-division}; hence its orbit is bounded by
		\[
		Ct^{-N}(1+|Y|+|\gamma|)^A.
		\]
		The signs $\varepsilon_G(\gamma)$ have absolute value one and do not alter this estimate.
		
		Let $\lambda_*>0$ be the shortest nonzero length in $\Lambda_G$. Polynomial lattice counting, applied separately to the finitely many cosets, gives
		\begin{equation*}
			\sum_{\gamma\in\Lambda_G\setminus\{0\}}
			(1+|\gamma|)^A e^{-|\gamma|^2/(4t)}
			\le Ct^{-N_0}e^{-\lambda_*^2/(8t)},
			\qquad 0<t\le1.
		\end{equation*}
		Indeed, absorb the polynomial into $t^{-N_0}e^{-|\gamma|^2/(8t)}$ and sum the remaining Gaussian; discreteness supplies the factor $e^{-\lambda_*^2/(8t)}$. Combining this estimate with the three orbit types above proves \eqref{zeta:ambient-small-time-bound}. The removable continuation in \Cref{thm:heat-kernel-image} extends the estimate across the root hyperplanes.
	\end{proof}
	
	\begin{theorem}\label{zeta:complete-local-heat}
		For every compact connected two-step pair $(G,L)$, there are $c>0$ and $N\ge0$ such that, as $t\downarrow0$,
		\begin{equation}\label{zeta:complete-local-heat-equation}
			\Tr\!\left(e^{t\Delta_{\pLie}}\right)
			=p_t(e)
			=C_{G,L}t^{-Q/2}e^{\beta_{G,L}t}
			+O\!\left(t^{-N}e^{-c/t}\right).
		\end{equation}
		Consequently, the coefficient of $t^{\ell-Q/2}$ in the heat trace expansion is
		\begin{equation}\label{zeta:all-heat-coefficients}
			C_{G,L}\frac{\beta_{G,L}^{\ell}}{\ell!},
			\qquad \ell\ge0.
		\end{equation}
	\end{theorem}
	
	\begin{proof}
		Insert \eqref{zeta:ambient-small-time-estimate} into \Cref{thm:integral-formula} at $g=e$. The identity
		\[
		\left(j_L^{\nc}(Y)/j_G^{\nc}(Y)\right)^{1/2}
		=\mathcal D_{G,L}(Y)
		\]
		and \eqref{zeta:local-constants} turn the zero lattice term into $C_{G,L}t^{-Q/2}e^{\beta_{G,L}t}$. By \eqref{zeta:relative-decay}, every polynomial moment of $\mathcal D_{G,L}$ is finite. Therefore \eqref{zeta:ambient-small-time-bound} remains $O(t^{-N'}e^{-c/t})$ after Hall inversion. This proves \eqref{zeta:complete-local-heat-equation}; expansion of $e^{\beta_{G,L}t}$ gives \eqref{zeta:all-heat-coefficients}.
	\end{proof}
	
	The reduced Mellin formula is
	\begin{equation}\label{zeta:mellin-formula}
		\zeta_{\pLie}(s)
		=\frac1{\Gamma(s)}\int_0^\infty
		t^{s-1}\bigl(p_t(e)-1\bigr)\,\mathrm dt,
		\qquad \Re s>Q/2.
	\end{equation}
	The heat calculus for positive Rockland operators gives meromorphic continuation and regularity at the nonpositive integers \cite[Corollary~2]{DaveHaller}; the preceding heat formula identifies all local terms explicitly.
	
	\begin{corollary}\label{zeta:local-zeta-data}
		The difference
		\begin{equation}\label{zeta:polar-model}
			\zeta_{\pLie}(s)
			-\frac{C_{G,L}}{\Gamma(s)}
			\sum_{\ell=0}^{\infty}
			\frac{\beta_{G,L}^{\ell}}
			{\ell!\,(s-Q/2+\ell)}
		\end{equation}
		is entire. Hence
		\begin{equation}\label{zeta:residues}
			\operatorname*{Res}_{s=Q/2-\ell}\zeta_{\pLie}(s)
			=\frac{C_{G,L}\beta_{G,L}^{\ell}}
			{\ell!\,\Gamma(Q/2-\ell)}
		\end{equation}
		whenever $Q/2-\ell$ is not a nonpositive integer. If $Q/2\in\mathbb N_0$, the only possible poles are the simple poles at
		\[
		s=1,\ldots,Q/2.
		\]
		Some of them may disappear when the corresponding local coefficient vanishes. In this case,
		\begin{align}
			\zeta_{\pLie}(0)
			&=C_{G,L}\frac{\beta_{G,L}^{Q/2}}{(Q/2)!}-1,
			\label{zeta:value-zero-even}\\
			\zeta_{\pLie}(-j)
			&=(-1)^j j!C_{G,L}
			\frac{\beta_{G,L}^{Q/2+j}}{(Q/2+j)!},
			\qquad j\ge1.
			\label{zeta:negative-values-even}
		\end{align}
		If $Q/2\in\frac12+\mathbb N_0$, then
		\begin{equation}\label{zeta:negative-values-odd}
			\zeta_{\pLie}(0)=-1,
			\qquad
			\zeta_{\pLie}(-j)=0\quad(j\ge1).
		\end{equation}
	\end{corollary}
	
	\begin{proof}
		Put
		\[
		E_0(t):=p_t(e)-C_{G,L}t^{-Q/2}e^{\beta_{G,L}t},
		\qquad
		E_\infty(t):=p_t(e)-1.
		\]
		By \eqref{zeta:complete-local-heat-equation}, there are $c>0$ and $N\ge0$ such that
		\[
		E_0(t)=O(t^{-N}e^{-c/t})\qquad(t\downarrow0).
		\]
		In particular, $E_0(t)=O(t^M)$ for every $M>0$. The spectral gap makes $E_\infty(t)$ exponentially decreasing at infinity. Consequently
		\[
		H_0(s):=\int_0^1t^{s-1}E_0(t)\,\mathrm dt,
		\qquad
		H_\infty(s):=\int_1^\infty t^{s-1}E_\infty(t)\,\mathrm dt
		\]
		are entire functions. Splitting \eqref{zeta:mellin-formula} at $t=1$ gives
		\begin{align*}
			\zeta_{\pLie}(s)
			&=\frac{H_0(s)+H_\infty(s)}{\Gamma(s)}
			+\frac{C_{G,L}}{\Gamma(s)}
			\int_0^1t^{s-Q/2-1}e^{\beta_{G,L}t}\,\mathrm dt
			-\frac{1}{s\Gamma(s)}.
		\end{align*}
		Expanding the exponential and integrating termwise first for $\Re s>Q/2$, then continuing meromorphically, yields
		\[
		\zeta_{\pLie}(s)
		=\frac{H_0(s)+H_\infty(s)}{\Gamma(s)}
		+\frac{C_{G,L}}{\Gamma(s)}
		\sum_{\ell\ge0}
		\frac{\beta_{G,L}^\ell}{\ell!\,(s-Q/2+\ell)}
		-\frac1{\Gamma(s+1)}.
		\]
		Let $S\subset\C$ be compact and disjoint from the candidate poles $Q/2-\mathbb N_0$. For all sufficiently large $\ell$,
		\[
		\inf_{s\in S}|s-Q/2+\ell|\ge\frac{\ell}{2}.
		\]
		Factorial decay then gives uniform convergence on $S$ of the displayed series and of every $s$-derivative; its finitely many initial summands are holomorphic on $S$. Thus, away from its simple candidate poles, the series and all of its $s$-derivatives converge uniformly on compact subsets and therefore define a meromorphic function. The first and last terms are entire, proving \eqref{zeta:polar-model}. Its residue at $s=Q/2-\ell$ is the one displayed in \eqref{zeta:residues} unless the zero of $1/\Gamma(s)$ cancels the pole. At $s=-j\in-\mathbb N_0$, call the local index $\ell$ \emph{resonant} when its denominator also vanishes there, that is, $-j-Q/2+\ell=0$. The factor $1/\Gamma(s)$ kills every nonresonant local summand and both entire Mellin transforms. A resonant summand occurs exactly when $Q/2\in\mathbb N_0$ and $\ell=Q/2+j$; using
		\[
		\lim_{s\to-j}\frac{1}{(s+j)\Gamma(s)}=(-1)^j j!
		\]
		gives \eqref{zeta:value-zero-even} and \eqref{zeta:negative-values-even}. The deleted zero mode is the explicit term $-1/\Gamma(s+1)$: it equals $-1$ at $s=0$ and vanishes at the negative integers. If $Q/2$ is half-integral there is no resonant summand, which gives \eqref{zeta:negative-values-odd}.
	\end{proof}
	
	The leading residue is $C_{G,L}/\Gamma(Q/2)$. Equivalently, the eigenvalue counting function satisfies
	\begin{equation*}
		\#\{j:\lambda_j\le\Lambda\}
		\sim \frac{C_{G,L}}{\Gamma(Q/2+1)}\Lambda^{Q/2}
		\qquad(\Lambda\to\infty),
	\end{equation*}
	by the Weyl law for positive Rockland operators \cite[Corollary~3]{DaveHaller}.
	
	We now analyze the global remainder terms. Define
	\begin{equation*}
		\det\nolimits_\zeta'(-\Delta_{\pLie})
		:=\exp\bigl(-\zeta_{\pLie}'(0)\bigr).
	\end{equation*}
	Let
	\[
	\gamma_{\mathrm E}
	:=\lim_{n\to\infty}\left(\sum_{k=1}^n\frac1k-\log n\right)
	\]
	be the Euler-Mascheroni constant. Set
	\begin{align*}
		\mathcal G_{G,L}
		&:=\int_0^1
		\frac{p_t(e)-C_{G,L}t^{-Q/2}e^{\beta_{G,L}t}}{t}\,\mathrm dt+\int_1^\infty\frac{p_t(e)-1}{t}\,\mathrm dt.
	\end{align*}
	Both integrals converge absolutely. For $a\ge0$, $\beta\in\R$, and $0<\tau\le1$, put
	\begin{equation*}
		\mathscr A_{a,\tau}(\beta):=
		\begin{cases}
			\displaystyle
			\sum_{\ell=0}^{\infty}
			\frac{\beta^\ell\tau^{\ell-a}}{\ell!(\ell-a)},
			& a\notin\mathbb N_0,\\[3mm]
			\displaystyle
			\sum_{\substack{\ell\ge0\\\ell\ne a}}
			\frac{\beta^\ell\tau^{\ell-a}}{\ell!(\ell-a)}
			+\frac{\beta^a}{a!}(\gamma_{\mathrm E}+\log\tau),
			& a\in\mathbb N_0.
		\end{cases}
	\end{equation*}
	Equivalently, after meromorphic continuation in $s$,
	\begin{equation}\label{zeta:local-mellin-formula}
		\mathscr A_{a,\tau}(\beta)
		=\left.\frac{\mathrm d}{\mathrm ds}\right|_{s=0}
		\left[
		\frac1{\Gamma(s)}\int_0^\tau t^{s-a-1}e^{\beta t}\,\mathrm dt
		\right].
	\end{equation}
	Thus $\mathscr A_{a,\tau}(\beta)$ is obtained by differentiating at $s=0$ the meromorphic continuation of the expression in brackets; the second line above is precisely the resonant case.
	
	\begin{proposition}
		For every compact connected two-step pair,
		\begin{equation}\label{zeta:local-global-determinant-formula}
			\log\det\nolimits_\zeta'(-\Delta_{\pLie})
			=\gamma_{\mathrm E}
			-C_{G,L}\mathscr A_{Q/2,1}(\beta_{G,L})
			-\mathcal G_{G,L}.
		\end{equation}
		Equivalently,
		\begin{equation}\label{zeta:finite-part-determinant}
			\log\det\nolimits_\zeta'(-\Delta_{\pLie})
			=-\operatorname*{FP}_{\varepsilon\downarrow0}
			\int_\varepsilon^\infty\frac{p_t(e)-1}{t}\,\mathrm dt
			-\gamma_{\mathrm E}\zeta_{\pLie}(0),
		\end{equation}
		where $\operatorname{FP}$ denotes the constant term after the power and logarithmic divergences prescribed by \eqref{zeta:complete-local-heat-equation} have been removed.
	\end{proposition}
	
	\begin{proof}
		Split \eqref{zeta:mellin-formula} at $t=1$ and subtract $C_{G,L}t^{-Q/2}e^{\beta_{G,L}t}$ on $(0,1)$. The zero mode contribution $-1/\Gamma(s+1)$ is kept separate and evaluated below. The remainder part is
		\[
		\frac1{\Gamma(s)}\left(
		\int_0^1t^{s-1}
		\bigl[p_t(e)-C_{G,L}t^{-Q/2}e^{\beta_{G,L}t}\bigr] \,\mathrm dt
		+\int_1^\infty t^{s-1}[p_t(e)-1] \,\mathrm dt
		\right).
		\]
		The bracket is holomorphic near zero and its value at $s=0$ is $\mathcal G_{G,L}$, so this part contributes $\mathcal G_{G,L}$ to $\zeta_{\pLie}'(0)$. Termwise integration of the subtracted exponential contributes $C_{G,L}\mathscr A_{Q/2,1}(\beta_{G,L})$.
		
		The zero mode term on $(0,1)$ is
		\[
		-\frac1{\Gamma(s)}\int_0^1t^{s-1}\,\mathrm dt
		=-\frac1{s\Gamma(s)}=-\frac1{\Gamma(s+1)}.
		\]
		Using
		\[
		\frac1{\Gamma(s)}=s+\gamma_{\mathrm E}s^2+O(s^3),
		\qquad
		-\frac1{\Gamma(s+1)}=-1-\gamma_{\mathrm E}s+O(s^2),
		\]
		shows that the remainder, local, and zero mode contributions to $\zeta_{\pLie}'(0)$ are respectively $\mathcal G_{G,L}$, $C_{G,L}\mathscr A_{Q/2,1}(\beta_{G,L})$, and $-\gamma_{\mathrm E}$. Taking the negative derivative gives \eqref{zeta:local-global-determinant-formula}; collecting the same singular terms before taking the constant part gives \eqref{zeta:finite-part-determinant}.
	\end{proof}
	
	Formula \eqref{zeta:local-global-determinant-formula} is the basic local-global decomposition. The constants $C_{G,L}$ and $\beta_{G,L}$ determine the subtraction at $t=0$, and $\mathcal G_{G,L}$ contains the finite contribution of the remaining heat trace.
	
	For a finite central covering, the zeta function and determinant decompose into character sectors. Let $\pi:\widetilde G\to G$ be a connected finite central local isometry covering with kernel $F$, and let $\widetilde L^\circ$ be the identity component of $\pi^{-1}(L)$. Right translation by $F$ gives
	\[
	L^2(\widetilde G)=\bigoplus_{\chi\in\widehat F}\mathcal H_\chi,
	\qquad
	\mathcal H_\chi
	:=\{f:f(xz)=\chi(z)^{-1}f(x)\}.
	\]
	Put
	\[
	A_\chi:=(-\widetilde\Delta_{\pLie})|_{\mathcal H_\chi}.
	\]
	For any nonnegative self-adjoint operator $A$ with discrete spectrum, write
	\[
	\zeta_A^{\mathrm{red}}(s):=\sum_{\mu\in\operatorname{Spec}(A),\,\mu>0}\mu^{-s},
	\qquad
	\det\nolimits_\zeta' A
	:=\exp\bigl(-({\zeta_A^{\mathrm{red}}})'(0)\bigr).
	\]
	The sector $A_1$ is unitarily equivalent to $-\Delta_{\pLie}$ on $G$, and the zero eigenspace lies in $\mathcal H_1$. For $\chi\ne1$, the space $\mathcal H_\chi$ contains no constants; the two-step condition identifies the full kernel with the constants, so $A_\chi$ is strictly positive and $\det_\zeta' A_\chi=\det_\zeta A_\chi$. Hence
	\begin{align*}
		\zeta_{-\widetilde\Delta_{\pLie}}^{\mathrm{red}}(s)
		&=\sum_{\chi\in\widehat F}\zeta_{A_\chi}^{\mathrm{red}}(s),\\
		\det\nolimits_\zeta'(-\widetilde\Delta_{\pLie})
		&=\prod_{\chi\in\widehat F}\det\nolimits_\zeta' A_\chi
		=\det\nolimits_\zeta'(-\Delta_{\pLie})
		\prod_{\chi\ne1}\det\nolimits_\zeta A_\chi.
	\end{align*}
	The zeta identity holds first on the common half-plane of absolute convergence and then by meromorphic continuation. Thus the determinant of the cover factors into the determinants of its character sectors.
	
	We next give an explicit formula for $\mathcal G_{G,L}$ when $G$ is simply connected. In this case $G$ is semisimple by \Cref{cor:heat-kernel-image-sc}; the lattice representation controls the small-time interval, the Peter-Weyl expansion controls the large-time interval, and a parameter $0<\tau\le1$ joins the two convergent formulas.
	
	Use the root data of $G$ fixed in \Cref{sec:preliminaries}. By \Cref{cor:heat-kernel-image-sc}, $\Lambda_G=2\pi Q_G^\vee$ and the lattice sign is identically one. Let
	\[
	\mathscr O_G:=(\Lambda_G\setminus\{0\})/\Wt_G
	\]
	be the set of nonzero Weyl orbits. For $\mathcal O\in\mathscr O_G$, let $\ell_{\mathcal O}$ be the common length of its elements. Recall from \eqref{zeta:weyl-constant} that $C_L=\vol(L)/(\vol(T)|\Wt_L|)$, and define
	\begin{align}
		a_{\mathcal O}(t)
		&:=C_L\int_{\tLie}\Pi_L(H)^2\mathcal D_{G,L}(H)\times
		\frac{\displaystyle\sum_{\gamma\in\mathcal O}
			\Pi_G(H-i\gamma)e^{-i\langle\gamma,H\rangle/(2t)}}
		{\Pi_G(H)}\,\mathrm dH.
		\label{zeta:orbit-coefficient}
	\end{align}
	
	\begin{proposition}
		
		The quotient in \eqref{zeta:orbit-coefficient} extends real analytically across the root hyperplanes. For each fixed $t>0$, the orbit-grouped image series converges absolutely and uniformly on compact subsets of the complex Cartan and is absolutely integrable after the Cartan reduction. There are integers $A,N\ge0$, depending only on $G$, and $C>0$ such that
		\begin{equation}\label{zeta:orbit-coefficient-bound}
			|a_{\mathcal O}(t)|
			\le Ct^{-N}(1+\ell_{\mathcal O})^A,
			\qquad 0<t\le1.
		\end{equation}
		Moreover,
		\begin{equation}\label{zeta:orbit-conjugacy}
			a_{-\mathcal O}(t)=\overline{a_{\mathcal O}(t)}.
		\end{equation}
		After the Mellin transform on $(0,\tau]$, the orbit series and all of its $s$-derivatives converge locally uniformly on $\C$.
	\end{proposition}
	
	\begin{proof}
		Write the numerator in \eqref{zeta:orbit-coefficient} as
		\[
		N_{\mathcal O,t}(H)
		:=\sum_{\gamma\in\mathcal O}
		\Pi_G(H-i\gamma)e^{-i\langle\gamma,H\rangle/(2t)}.
		\]
		For $w\in\Wt_G$, replace $\gamma$ by $w\gamma$ in the sum. Weyl invariance of the orbit and inner product, together with $\Pi_G(wX)=\det(w)\Pi_G(X)$, gives
		\[
		N_{\mathcal O,t}(wH)=\det(w)N_{\mathcal O,t}(H).
		\]
		Thus the numerator is Weyl anti-invariant and its quotient by $\Pi_G$ extends across the root hyperplanes. For real $H$, the factors $\Pi_L(H)^2\mathcal D_{G,L}(H)$ and $\Pi_G(H)$ are real, while
		\[
		N_{-\mathcal O,t}(H)=\overline{N_{\mathcal O,t}(H)}.
		\]
		Integration proves \eqref{zeta:orbit-conjugacy}.
		
		Every derivative of $N_{\mathcal O,t}$ either differentiates the polynomial $\Pi_G(H-i\gamma)$ or the phase, in the latter case producing a factor $\gamma/(2t)$. Since an orbit has at most $|\Wt_G|$ elements, \eqref{zeta:weyl-division-bound} therefore gives, on the real Cartan,
		\[
		\left|\frac{N_{\mathcal O,t}(H)}{\Pi_G(H)}\right|
		\le Ct^{-N}(1+|H|+\ell_{\mathcal O})^A
		\]
		for integers $A,N$ depending only on $G$. The exponential decay in \eqref{zeta:relative-decay} absorbs the polynomial in $H$, so integration against $\Pi_L(H)^2\mathcal D_{G,L}(H)$ proves \eqref{zeta:orbit-coefficient-bound}. On a compact subset $S$ of the complexified Cartan subalgebra $\tLie_{G,\C}$, the same iterated Weyl division calculation gives
		\[
		C_S t^{-N}(1+\ell_{\mathcal O})^A
		e^{C_S\ell_{\mathcal O}/t}.
		\]
		After multiplication by the lattice Gaussian, the exponent satisfies
		\[
		-\frac{\ell_{\mathcal O}^2}{4t}
		+\frac{C_S\ell_{\mathcal O}}{t}
		\le -\frac{\ell_{\mathcal O}^2}{8t}
		\qquad\text{when }\ell_{\mathcal O}\ge8C_S.
		\]
		Only finitely many orbits fail this inequality. Polynomial lattice counting therefore gives uniform convergence on $S$ for each fixed $t>0$.
		
		Finally, let $S_0\subset\C$ be compact, put $\sigma:=\min_{s\in S_0}\Re s$, and fix $q\ge0$. On the real Cartan the complex growth factor is absent. Using \eqref{zeta:orbit-coefficient-bound}, polynomial lattice counting, and the shortest nonzero lattice length, the sum of the $q$-times differentiated Mellin integrands is bounded by
		\begin{equation}\label{zeta:lattice-mellin-majorant}
			C_{S_0,q}\,
			t^{\sigma-Q/2-N-1}(1+|\log t|)^q e^{-c/t},
			\qquad 0<t\le\tau,
		\end{equation}
		which is integrable at zero. Dominated convergence, uniformly for $s\in S_0$, proves local uniform convergence and justifies every termwise $s$-derivative.
	\end{proof}
	
	Pair the nonzero Weyl orbits under inversion and write
	\[
	\mathscr P_G:=\mathscr O_G/(\mathcal O\sim-\mathcal O).
	\]
	For $\mathcal P\in\mathscr P_G$, choose a representative $\mathcal O\in\mathcal P$, set $\ell_{\mathcal P}:=\ell_{\mathcal O}$, and define
	\begin{equation}\label{zeta:paired-orbit-coefficient}
		b_{\mathcal P}(t):=
		\begin{cases}
			a_{\mathcal O}(t),&\mathcal O=-\mathcal O,\\
			a_{\mathcal O}(t)+a_{-\mathcal O}(t),&\mathcal O\ne-\mathcal O.
		\end{cases}
	\end{equation}
	The definition is independent of the representative, $b_{\mathcal P}(t)\in\mathbb R$, and
	\[
	|b_{\mathcal P}(t)|
	\le 2Ct^{-N}(1+\ell_{\mathcal P})^A.
	\]
	
	\begin{theorem}
		Under the simply connected hypotheses above, for every $t>0$,
		\begin{align}
			p_t(e)
			&=C_{G,L}t^{-Q/2}e^{\beta_{G,L}t}
			+\frac{\vol(G)}{(4\pi)^{Q/2}}t^{-Q/2}e^{\beta_{G,L}t}
			\sum_{\mathcal P\in\mathscr P_G}
			e^{-\ell_{\mathcal P}^2/(4t)}b_{\mathcal P}(t).
			\label{zeta:exact-heat-formula}
		\end{align}
		The series is real and absolutely convergent; more precisely,
		\[
		\sum_{\mathcal P\in\mathscr P_G}
		e^{-\ell_{\mathcal P}^2/(4t)}|b_{\mathcal P}(t)|<\infty.
		\]
	\end{theorem}
	
	\begin{proof}
		Specialize \eqref{eq:ambient-complex-image-sc} to $Z=iH$, multiply by $j_G^{\nc}(H)^{1/2}e^{-|H|^2/(4t)}$, and first group the nonzero lattice vectors by Weyl orbit. Pairing each orbit with its negative through \eqref{zeta:orbit-conjugacy} gives the real coefficients \eqref{zeta:paired-orbit-coefficient}. Inserting the resulting identity into \Cref{thm:integral-formula} and applying \eqref{eq:weyl-lie} proves \eqref{zeta:exact-heat-formula}. The zero orbit gives the first term, and \eqref{zeta:orbit-coefficient-bound} together with polynomial lattice counting proves absolute convergence.
	\end{proof}
	
	For $\lambda\in\widehat G$, write $V_\lambda$ for the irreducible $G$-module and $d_\lambda=\dim V_\lambda$. Its restriction to $L$ is
	\[
	V_\lambda|_L
	\cong\bigoplus_{\nu\in\widehat L}
	\mathbb C^{m_{\lambda,\nu}}\otimes V_\nu,
	\qquad
	m_{\lambda,\nu}:=\dim\Hom_L(V_\nu,V_\lambda),
	\]
	where $d_\nu=\dim V_\nu$. With
	\[
	\Omega_G:=\sum_{a=1}^rY_a^2+\sum_{i=1}^mX_i^2,
	\qquad
	\Omega_L:=\sum_{a=1}^rY_a^2,
	\]
	normalize the nonnegative Casimir eigenvalues by
	\[
	\mathrm d\pi_\lambda(\Omega_G)=-c_G(\lambda)I,
	\qquad
	\mathrm d\pi_\nu(\Omega_L)=-c_L(\nu)I.
	\]
	On the block $\mathbb C^{m_{\lambda,\nu}}\otimes V_\nu$, the operator $-\Delta_{\pLie}$ has eigenvalue
	\[
	\Lambda_{\lambda,\nu}:=c_G(\lambda)-c_L(\nu)\ge0.
	\]
	Peter-Weyl contributes the additional multiplicity $d_\lambda$, so the scalar multiplicity of this block is $d_\lambda m_{\lambda,\nu}d_\nu$. Consequently,
	\begin{equation}\label{zeta:spectral-heat-trace}
		p_t(e)-1
		=\sum_{\lambda\in\widehat G}
		\sum_{\nu\in\widehat L}^{\prime}
		d_\lambda m_{\lambda,\nu}d_\nu
		e^{-t\Lambda_{\lambda,\nu}}.
	\end{equation}
	Here the prime removes the unique term with $\Lambda_{\lambda,\nu}=0$. All summands are nonnegative, and \eqref{zeta:heat-trace-identity} gives absolute convergence. The two-step condition makes the constant functions the whole kernel, so this is precisely the unique trivial block. Let
	\[
	\Gamma(s,x):=\int_x^\infty u^{s-1}e^{-u}\,\mathrm du
	\]
	denote the upper incomplete Gamma function. For every compact set $S\subset\C$, integer $q\ge0$, and $x_0>0$,
	\begin{equation}\label{zeta:incomplete-gamma-bound}
		\sup_{s\in S}
		\left|
		\partial_s^q\!\left(x^{-s}\Gamma(s,x)\right)
		\right|
		\le C_{S,q,x_0}e^{-x},
		\qquad x\ge x_0.
	\end{equation}
	Indeed, $x^{-s}\Gamma(s,x)=\int_1^\infty v^{s-1}e^{-xv}\,\mathrm dv$, and differentiation under the integral proves the bound.
	
	Splitting \eqref{zeta:mellin-formula} at $\tau$, using \eqref{zeta:exact-heat-formula} below $\tau$, and using \eqref{zeta:spectral-heat-trace} above $\tau$ yields
	\begin{align}
		\zeta_{\pLie}(s)
		&=\frac{C_{G,L}}{\Gamma(s)}
		\int_0^\tau t^{s-Q/2-1}e^{\beta_{G,L}t}\,\mathrm dt
		-\frac{\tau^s}{\Gamma(s+1)}
		\notag\\
		&\quad+\frac{\vol(G)}{(4\pi)^{Q/2}\Gamma(s)}
		\sum_{\mathcal P\in\mathscr P_G}
		\int_0^\tau t^{s-Q/2-1}
		e^{\beta_{G,L}t-\ell_{\mathcal P}^2/(4t)}
		b_{\mathcal P}(t)\,\mathrm dt
		\notag\\
		&\quad+\frac1{\Gamma(s)}
		\sum_{\lambda\in\widehat G}
		\sum_{\nu\in\widehat L}^{\prime}
		d_\lambda m_{\lambda,\nu}d_\nu
		\Lambda_{\lambda,\nu}^{-s}
		\Gamma(s,\tau\Lambda_{\lambda,\nu}).
		\label{zeta:global-splitting}
	\end{align}
	Initially \eqref{zeta:global-splitting} holds for $\Re s>Q/2$. The first term has the meromorphic expansion
	\[
	\frac{C_{G,L}}{\Gamma(s)}
	\sum_{\ell\ge0}
	\frac{\beta_{G,L}^{\ell}\tau^{s-Q/2+\ell}}
	{\ell!\,(s-Q/2+\ell)},
	\]
	and consequently its difference from $\zeta_{\pLie}(s)$ extends to an entire function, as in \eqref{zeta:polar-model}. The zero mode term is entire, and the lattice term is entire by \eqref{zeta:lattice-mellin-majorant}. For the spectral term, let $\Lambda_1>0$ be the first positive eigenvalue. Writing
	\[
	\Lambda^{-s}\Gamma(s,\tau\Lambda)
	=\tau^s(\tau\Lambda)^{-s}\Gamma(s,\tau\Lambda)
	\]
	and applying \eqref{zeta:incomplete-gamma-bound} shows that every fixed number of $s$-derivatives is bounded, uniformly on compact $s$-sets, by a constant times $e^{-\tau\Lambda}$. The resulting summable majorant is exactly
	\begin{equation}\label{zeta:spectral-derivative-majorant}
		\sum_{\lambda\in\widehat G}
		\sum_{\nu\in\widehat L}^{\prime}
		d_\lambda m_{\lambda,\nu}d_\nu e^{-\tau\Lambda_{\lambda,\nu}}
		=p_\tau(e)-1<\infty.
	\end{equation}
	Thus, for every $q\ge0$, the spectral series and its first $q$ $s$-derivatives converge uniformly on compact subsets; in particular, the series defines an entire function. Equation \eqref{zeta:global-splitting} is therefore an exact continuation formula; its derivation from a split integral also makes it independent of $\tau$.
	
	Let
	\[
	E_1(x):=\Gamma(0,x)=\int_x^\infty\frac{e^{-u}}u\,\mathrm du
	\]
	denote the classical exponential integral. Define the two convergent remainders
	\begin{align}
		\mathcal R_{\mathrm{lat}}(\tau)
		&:=\frac{\vol(G)}{(4\pi)^{Q/2}}
		\sum_{\mathcal P\in\mathscr P_G}
		\int_0^\tau t^{-Q/2-1}
		e^{\beta_{G,L}t-\ell_{\mathcal P}^2/(4t)}
		b_{\mathcal P}(t)\,\mathrm dt,
		\label{zeta:lattice-determinant-term}\\
		\mathcal R_{\mathrm{spec}}(\tau)
		&:=\sum_{\lambda\in\widehat G}
		\sum_{\nu\in\widehat L}^{\prime}
		d_\lambda m_{\lambda,\nu}d_\nu
		E_1(\tau\Lambda_{\lambda,\nu}).
		\label{zeta:spectral-determinant-term}
	\end{align}
	\begin{theorem}\label{zeta:exact-determinant}
		If $G$ is simply connected, then, for every $0<\tau\le1$,
		\begin{align}
			\log\det\nolimits_\zeta'(-\Delta_{\pLie})
			&=\gamma_{\mathrm E}+\log\tau
			-C_{G,L}\mathscr A_{Q/2,\tau}(\beta_{G,L})
			\notag\\
			&\quad-\mathcal R_{\mathrm{lat}}(\tau)
			-\mathcal R_{\mathrm{spec}}(\tau).
			\label{zeta:exact-determinant-formula}
		\end{align}
		Every term on the right converges absolutely, and the expression is independent of $\tau$.
	\end{theorem}
	
	\begin{proof}
		The uniform convergence on compact subsets established above permits differentiation of \eqref{zeta:global-splitting} at $s=0$. For a nonresonant index $\ell\ne Q/2$, the derivative of the corresponding local summand is
		\[
		\left.\frac{\mathrm d}{\mathrm ds}\right|_{s=0}
		\frac1{\Gamma(s)}
		\frac{\beta_{G,L}^{\ell}\tau^{s-Q/2+\ell}}
		{\ell!\,(s-Q/2+\ell)}
		=\frac{\beta_{G,L}^{\ell}\tau^{\ell-Q/2}}
		{\ell!\,(\ell-Q/2)}.
		\]
		If $Q/2\in\mathbb N_0$, the resonant summand $\ell=Q/2$ instead uses
		\begin{equation*}
			\left.\frac{\mathrm d}{\mathrm ds}\right|_{s=0}
			\frac{\tau^s}{s\Gamma(s)}
			=\gamma_{\mathrm E}+\log\tau.
		\end{equation*}
		Thus the local derivative is $C_{G,L}\mathscr A_{Q/2,\tau}(\beta_{G,L})$.
		
		Since $1/\Gamma(s)=s+O(s^2)$, the lattice and spectral derivatives are \eqref{zeta:lattice-determinant-term} and \eqref{zeta:spectral-determinant-term}; their termwise differentiation is justified by \eqref{zeta:lattice-mellin-majorant} and \eqref{zeta:spectral-derivative-majorant}. Finally, the zero mode deletion on $(0,\tau)$ is
		\[
		-\frac1{\Gamma(s)}\int_0^\tau t^{s-1}\,\mathrm dt
		=-\frac{\tau^s}{\Gamma(s+1)}.
		\]
		Its derivative at zero is $-(\gamma_{\mathrm E}+\log\tau)$, so it contributes $-(\gamma_{\mathrm E}+\log\tau)$ to $\zeta_{\pLie}'(0)$. Negating the derivative proves \eqref{zeta:exact-determinant-formula}.
	\end{proof}
	
	These estimates give quantitative truncation errors. There are constants $C>0$ and $N_1\ge0$, depending only on $(G,L)$, such that, uniformly for $0<\tau\le1$ and all truncation levels $R_{\mathrm{lat}},R_{\mathrm{spec}}\ge0$, retaining the paired orbits with $\ell_{\mathcal P}\le R_{\mathrm{lat}}$ changes \eqref{zeta:lattice-determinant-term} by at most
	\begin{equation}\label{zeta:lattice-tail}
		C\tau^{-Q/2-N_1}
		\exp\!\left(-\frac{R_{\mathrm{lat}}^2}{8\tau}\right),
	\end{equation}
	and retaining the spectral terms with $\Lambda_{\lambda,\nu}\le R_{\mathrm{spec}}$ changes \eqref{zeta:spectral-determinant-term} by at most
	\begin{equation}\label{zeta:spectral-tail}
		C\tau^{-Q/2-1}e^{-\tau R_{\mathrm{spec}}/2}.
	\end{equation}
	
	To prove \eqref{zeta:lattice-tail}, combine \eqref{zeta:orbit-coefficient-bound} with
	\begin{equation*}
		(1+\ell)^A e^{-\ell^2/(4t)}
		\le C_A t^{-A/2}e^{-\ell^2/(8t)}e^{-\ell^2/(16t)},
		\qquad 0<t\le1.
	\end{equation*}
	For every omitted orbit, the first Gaussian on the right is bounded by $e^{-R_{\mathrm{lat}}^2/(8\tau)}$. Polynomial lattice counting controls the sum of the second Gaussian, while the shortest nonzero lattice length makes the resulting integral at $t=0$ convergent. Increasing the fixed power of $\tau^{-1}$ gives \eqref{zeta:lattice-tail}.
	
	For the spectral tail, let $\Lambda_1>0$ be the first positive eigenvalue. Since $E_1(x)\le e^{-x}/x$, every omitted term satisfies
	\[
	E_1(\tau\Lambda)
	\le \frac{e^{-\tau R_{\mathrm{spec}}/2}}
	{\tau\Lambda_1}e^{-\tau\Lambda/2}.
	\]
	Consequently,
	\begin{align*}
		&\sum_{\substack{\lambda\in\widehat G,\ \nu\in\widehat L\\
				\Lambda_{\lambda,\nu}>R_{\mathrm{spec}}}}
		d_\lambda m_{\lambda,\nu}d_\nu
		E_1(\tau\Lambda_{\lambda,\nu})\le
		\frac{e^{-\tau R_{\mathrm{spec}}/2}}{\tau\Lambda_1}
		\bigl(p_{\tau/2}(e)-1\bigr)
		\le C\tau^{-Q/2-1}e^{-\tau R_{\mathrm{spec}}/2},
	\end{align*}
	which is \eqref{zeta:spectral-tail}. The parameter $\tau$ can therefore be chosen by balancing the exponents $R_{\mathrm{lat}}^2/(8\tau)$ and $\tau R_{\mathrm{spec}}/2$. For positive cutoffs, the balanced choice subject to $0<\tau\le1$ is
	\begin{equation}\label{zeta:balanced-splitting-time}
		\tau_{\mathrm{bal}}
		:=\min\left\{1,\frac{R_{\mathrm{lat}}}
		{2\sqrt{R_{\mathrm{spec}}}}\right\}.
	\end{equation}
	This balances the exponential factors; the displayed polynomial prefactors remain unchanged.
	
	\section{Examples: block families and symmetric specializations}
	\label{sec:examples}
	
	This section specializes the abstract results to block subgroups of $\SU(N)$ and then to two-block compact symmetric pairs. We give the bracket structure, root data, local constants, vertical coefficient, and classical spherical reductions in the same normalizations as the preceding sections.
	
	Consider the block family with an arbitrary connected subtorus in its block scalar factor. In this part set $\theta=\mathrm{id}$, so $K=G=\SU(N)$, and reserve $L$ for the subgroup defining the horizontal complement. The symmetric cases are then obtained by changing the involution and setting $L=K$.
	
	Fix a partition
	\[
	N=n_1+\cdots+n_s,
	\qquad s\ge2,
	\qquad n_a\ge1,
	\]
	and put $\mathbf n=(n_1,\ldots,n_s)$. Define the traceless block scalar algebra and its connected torus by
	\begin{align*}
		\mathfrak z_{\mathbf n}
		&:={}
		\left\{\diag(i\beta_1I_{n_1},\ldots,i\beta_sI_{n_s}):
		\beta_a\in\mathbb R,\ \sum_{a=1}^s n_a\beta_a=0\right\},\\
		Z_{\mathbf n}&:=\exp(\mathfrak z_{\mathbf n}).
	\end{align*}
	Let $T_0\subset Z_{\mathbf n}$ be a closed connected subtorus, with Lie algebra $\mathfrak t_0$, and set
	\begin{equation}\label{eq:examples-block-subgroup}
		L_{\mathbf n,T_0}
		:=\left(\prod_{a=1}^s\SU(n_a)\right)T_0\subset\SU(N),
		\qquad
		\lLie_{\mathbf n,T_0}
		=\left(\bigoplus_{a=1}^s\mathfrak{su}(n_a)\right)
		\oplus\mathfrak t_0.
	\end{equation}
	All metrics below are induced by $\langle X,Y\rangle=-\Tr(XY)$.
	
	Let
	\[
	V_{\mathbf n}
	:=\left\{\beta\in\mathbb R^s:\sum_{a=1}^s n_a\beta_a=0\right\},
	\qquad
	\langle\beta,\gamma\rangle_{\mathbf n}
	:=\sum_{a=1}^s n_a\beta_a\gamma_a,
	\qquad \beta,\gamma\in V_{\mathbf n}\subset\mathbb R^s.
	\]
	Let $V_0\subset V_{\mathbf n}$ correspond to $\mathfrak t_0$. The horizontal space $\pLie=\lLie_{\mathbf n,T_0}^{\perp}\subset\mathfrak{su}(N)$ consists of the matrices
	\begin{equation*}
		X(\beta,Z)=
		\begin{pmatrix}
			i\beta_1I_{n_1}&Z_{12}&\cdots&Z_{1s}\\
			-Z_{12}^*&i\beta_2I_{n_2}&\cdots&Z_{2s}\\
			\vdots&\vdots&\ddots&\vdots\\
			-Z_{1s}^*&-Z_{2s}^*&\cdots&i\beta_sI_{n_s}
		\end{pmatrix},
		\quad
		\beta\in V_0^\perp,\qquad
		Z_{ab}\in M_{n_a\times n_b}(\mathbb C),
	\end{equation*}
	and
	\begin{equation*}
		|X(\beta,Z)|^2
		=\sum_{a=1}^s n_a\beta_a^2
		+2\sum_{a<b}\|Z_{ab}\|_{\mathrm{HS}}^2.
	\end{equation*}
	
	\begin{proposition}
		For every $\mathbf n$ and $T_0$ above, the pair $(\SU(N),L_{\mathbf n,T_0})$ is a two-step pair:
		\[
		\mathfrak{su}(N)=\pLie+[\pLie,\pLie].
		\]
	\end{proposition}
	
	\begin{proof}
		All off-diagonal block matrices belong to $\pLie$. If $p,q$ lie in block $a$ and $\mu$ lies in a different block, set
		\[
		U_{p\mu}=E_{p\mu}-E_{\mu p},
		\qquad
		V_{p\mu}=\mathrm i(E_{p\mu}+E_{\mu p}).
		\]
		Their brackets contain $E_{qp}-E_{pq}$, $\mathrm i(E_{pq}+E_{qp})$, and the differences $2\mathrm i(E_{pp}-E_{qq})$; hence they span every $\mathfrak{su}(n_a)$. Moreover, if $P_a$ is the coordinate projection onto block $a$, then
		\[
		\sum_{p\in I_a}\sum_{\mu\in I_b}[U_{p\mu},V_{p\mu}]
		=2\mathrm i(n_b P_a-n_a P_b).
		\]
		These elements span $\mathfrak z_{\mathbf n}$. Thus $[\pLie,\pLie]$ contains the full block-diagonal algebra, while $\pLie$ contains every off-diagonal block, proving the claim.
	\end{proof}
	
	Here is the root data needed by all of the general theorems. A maximal torus $T\subset L_{\mathbf n,T_0}$ has Lie algebra
	\[
	\tLie=\left(\bigoplus_{a=1}^s\tLie_a\right)\oplus\mathfrak t_0,
	\]
	where an element can be written
	\[
	H=\diag\bigl(i(x_1^{(1)}+\xi_1),\ldots,
	i(x_{n_s}^{(s)}+\xi_s)\bigr),
	\quad
	\sum_{i=1}^{n_a}x_i^{(a)}=0,\qquad \xi\in V_0.
	\]
	Its exponential lattice is
	\begin{equation*}
		\Lambda_T=\ker(\exp\colon\tLie\to T)
		=\left\{H:x_i^{(a)}+\xi_a\in2\pi\mathbb Z
		\text{ for every }a,i\right\}.
	\end{equation*}
	For the ambient diagonal torus of $\SU(N)$ this becomes
	\begin{equation}\label{eq:examples-su-lattice}
		\Lambda_{\SU(N)}
		=\left\{2\pi i\diag(k_1,\ldots,k_N):
		k_j\in\mathbb Z,\ \sum_{j=1}^Nk_j=0\right\}
		=2\pi Q_{A_{N-1}}^\vee,
	\end{equation}
	where $A_{N-1}$ denotes the root system of $\mathfrak{su}(N)$ and $Q_{A_{N-1}}^\vee$ its coroot lattice. The positive roots of $L_{\mathbf n,T_0}$, its Weyl group, and its Weyl polynomial are
	\begin{align*}
		R_L^+
		&=\coprod_{a=1}^s
		\{x_i^{(a)}-x_j^{(a)}:1\le i<j\le n_a\},
		\\
		\Wt_L&=\prod_{a=1}^s\mathfrak S_{n_a},
		\qquad |\Wt_L|=\prod_{a=1}^sn_a!,\\
		\Pi_L(H)&=\prod_{a=1}^s\prod_{i<j}
		(x_i^{(a)}-x_j^{(a)}).
	\end{align*}
	Here $\mathfrak S_{n_a}$ denotes the symmetric group on $n_a$ letters. The nonzero isotropy weights are the nonzero restrictions, counted with multiplicity, of
	\begin{equation}\label{eq:examples-block-isotropy-weights}
		\omega_{ij}^{ab}(H)
		:=x_i^{(a)}-x_j^{(b)}+\xi_a-\xi_b,
		\qquad a<b,\qquad 1\le i\le n_a,\qquad1\le j\le n_b.
	\end{equation}
	The block scalar summand corresponding to $V_0^\perp$, and any vanishing restriction in \eqref{eq:examples-block-isotropy-weights}, contributes only zero weights. Consequently
	\begin{equation*}
		\mathcal D_{\SU(N),L_{\mathbf n,T_0}}(H)
		=\prod_{a<b}\prod_{i=1}^{n_a}\prod_{j=1}^{n_b}
		w\bigl(\omega_{ij}^{ab}(H)\bigr),
	\end{equation*}
	with $w(0)=1$.
	
	Write
	\[
	d_L=\sum_{a=1}^s(n_a^2-1)+\dim T_0,
	\qquad Q=N^2-1+d_L.
	\]
	The heat kernel inversion formula \Cref{thm:integral-formula} gives, for $g\in\SU(N)$,
	\begin{align*}
		p_t(g)
		&={}
		\frac{\exp\!\left[-\dfrac{t}{12}
			\sum_{a=1}^sn_a(n_a^2-1)\right]}
		{(4\pi t)^{d_L/2}}
		\times
		\int_{\lLie_{\mathbf n,T_0}}
		q_t^{\C}(ge^{iY})j_L^{\nc}(Y)^{1/2}
		e^{-|Y|^2/(4t)}\,\mathrm dY,
	\end{align*}
	where, on the Cartan above,
	\[
	j_L^{\nc}(H)^{1/2}
	=\prod_{a=1}^s\prod_{i<j}
	\frac{\sinh((x_i^{(a)}-x_j^{(a)})/2)}
	{(x_i^{(a)}-x_j^{(a)})/2}.
	\]
	The image formula in \Cref{cor:heat-kernel-image-sc}, together with \eqref{eq:examples-su-lattice}, makes $q_t^{\C}$ an explicit coroot lattice theta quotient.
	
	With the notation in \eqref{zeta:local-constants}, \Cref{zeta:complete-local-constants} gives
	\begin{align}
		\beta_{\SU(N),L_{\mathbf n,T_0}}
		&=\frac1{12}\left[
		N(N^2-1)-\sum_{a=1}^sn_a(n_a^2-1)\right],
		\notag\\
		C_{\SU(N),L_{\mathbf n,T_0}}
		&=\frac{\vol(\SU(N))\vol(L_{\mathbf n,T_0})}
		{(4\pi)^{Q/2}\vol(T)\prod_{a=1}^sn_a!}
		\notag\\
		&\quad\times
		\int_{\tLie}\Pi_L(H)^2
		\prod_{a<b}\prod_{i,j}w(\omega_{ij}^{ab}(H))\,\mathrm dH.
		\label{eq:examples-block-C}
	\end{align}
	
	Two endpoints of the family are useful to keep distinct. Taking $T_0=\{e\}$ gives the semisimple block subgroup $\prod_a\SU(n_a)$; in particular, $(n_1,n_2)=(N-1,1)$ gives $L=\SU(N-1)$. Taking $T_0=Z_{\mathbf n}$ gives the full block Levi subgroup
	\[
	S\bigl(\U(n_1)\times\cdots\times\U(n_s)\bigr).
	\]
	Both are covered by the preceding formulas, but only the following two-block full Levi case is symmetric.
	
	Assume now $s=2$ and $T_0=Z_{\mathbf n}$. Define
	\[
	\theta(g)=\diag(I_{n_1},-I_{n_2})\,g\,
	\diag(I_{n_1},-I_{n_2})^{-1}.
	\]
	Then
	\[
	K=(G^\theta)_0=S(\U(n_1)\times\U(n_2)),
	\qquad L=K,
	\]
	is the compact symmetric pair of type AIII in Cartan's classification \cite[Table~V]{Helgason}. The one-column specialization below is also the sub-Riemannian $K+P$ geometry used to formulate time-optimal control for a class of multilevel quantum systems \cite[Sec.~I]{AlbertiniDAlessandroSheller2020}. Here $\pLie$ consists only of the two off-diagonal blocks, and $[\pLie,\pLie]=\kLie$; hence every result of \Cref{sec:sp-vertical} applies. In particular, the vertical coefficient is the attained singular-value minimum \eqref{eq:sp-F}, satisfies the finite-mesh bounds \eqref{eq:sp-grid-enclosure}, and controls the compact endpoint with the remainder in \eqref{eq:sp-compact-vertical}. These formulas apply to every two-block full Levi pair.
	
	The one-column specialization has a stronger conclusion. Let $n\ge1$,
	\begin{equation}\label{eq:examples-hermitian-pair}
		G=\SU(n+1),
		\qquad
		K=L=S(\U(n)\times\U(1))\cong\U(n).
	\end{equation}
	For $n=1$, this is, up to the normalization of the invariant metric and heat time, the canonical contact sub-Laplacian on $\SU(2)$ studied by Baudoin and Bonnefont \cite{BaudoinBonnefont2009}. Define
	\begin{align*}
		H_M&:=i\diag(M,-\Tr M),\qquad M\in\Herm(n),
		\\
		X_z&:=\begin{pmatrix}0&z\\-z^*&0\end{pmatrix},
		\qquad z\in\mathbb C^n,\qquad |X_z|^2=2\|z\|^2.
	\end{align*}
	If $\lambda_1^+(M)\ge\cdots\ge\lambda_{p(M)}^+(M)>0$ are the positive eigenvalues of $M$, and $\lambda_1^-(M)\ge\cdots\ge\lambda_{q(M)}^-(M)>0$ are the absolute values of its negative eigenvalues, put
	\begin{equation}\label{eq:sp-Phi}
		\Phi_n(M)
		:=4\pi\left(
		\sum_{j=1}^{p(M)}j\lambda_j^+(M)
		+\sum_{j=1}^{q(M)}j\lambda_j^-(M)
		\right).
	\end{equation}
	
	\begin{theorem}\label{thm:examples-AIII}
		For the pair \eqref{eq:examples-hermitian-pair},
		\begin{equation}\label{eq:sp-AIII-coefficient}
			\mathfrak F_{G,K}(H_M)=\Phi_n(M).
		\end{equation}
		An exact minimizing loop uses only frequencies $|k|\le n$, and, uniformly for $\Phi_n(M)\le R$,
		\begin{equation*}
			d_{\sr}(e,\exp(\varepsilon H_M))^2
			=\varepsilon\Phi_n(M)+O_R(\varepsilon^{3/2}),
		\end{equation*}
		with lower error $O_R(\varepsilon^2)$.
	\end{theorem}
	
	\begin{proof}
		Direct multiplication shows $[\pLie,\pLie]=\kLie$. With the area and energy notation of \eqref{eq:sp-loop-isoperimetry}, a based path $z:[0,1]\to\mathbb C^n$ satisfies
		\begin{align*}
			\mathcal A_{G,K}(X_z)&=H_{\mathcal A_n(z)},\\
			\mathcal A_n(z)&=-\frac{\mathrm i}{2}
			\int_0^1(\dot z z^*-z\dot z^*)\,\mathrm ds,\\
			\mathcal E(X_z)&=2\int_0^1\|\dot z(s)\|^2\,\mathrm ds.
		\end{align*}
		Write $z(s)=\sum_{k\in\mathbb Z}z_k e^{2\pi i k s}$. Parseval's identity gives
		\begin{equation}\label{eq:sp-AIII-Fourier}
			\mathcal A_n(z)=2\pi\sum_k k z_kz_k^*,
			\qquad
			\mathcal E(X_z)=8\pi^2\sum_k k^2\|z_k\|^2.
		\end{equation}
		For $k\ge1$, write
		\[
		z_k^+:=z_k,\qquad z_k^-:=z_{-k},\qquad
		P_k^\pm:=2\pi k z_k^\pm(z_k^\pm)^*,
		\qquad P_\pm:=\sum_{k\ge1}P_k^\pm.
		\]
		For the lower bound, let $z$ be any based path satisfying $\mathcal A_n(z)=M$. Then
		\[
		M=P_+-P_-,
		\qquad
		\mathcal E(X_z)=4\pi\sum_{k\ge1}k
		(\Tr P_k^++\Tr P_k^-).
		\]
		For any positive rank-one decomposition $P_+=\sum_k P_k^+$, the Ky Fan variational principle \cite[Chapter~III, \S1]{Bhatia1997}, applied to the tails $\sum_{k>\ell}P_k^+$, yields
		\[
		\sum_{k\ge1}k\Tr P_k^+
		\ge\sum_{j=1}^n j\lambda_j(P_+)
		\ge\sum_{j=1}^{p(M)}j\lambda_j^+(M).
		\]
		Applying the argument to $P_-$ and $-M$ proves $\mathcal E(X_z)\ge\Phi_n(M)$.
		
		Choose orthonormal eigenvectors $u_j^\pm$ of the positive and negative parts of $M$. Equality is attained by
		\begin{equation}\label{eq:sp-AIII-minimizer}
			z_M(s)=
			\sum_{j=1}^{p(M)}
			\sqrt{\frac{\lambda_j^+(M)}{2\pi j}}
			(e^{2\pi ijs}-1)u_j^+
			+\sum_{j=1}^{q(M)}
			\sqrt{\frac{\lambda_j^-(M)}{2\pi j}}
			(e^{-2\pi ijs}-1)u_j^-.
		\end{equation}
		Indeed, \eqref{eq:sp-AIII-Fourier} gives $\mathcal A_n(z_M)=M$ and $\mathcal E(X_{z_M})=\Phi_n(M)$. The loop formula \eqref{eq:sp-loop-isoperimetry} proves \eqref{eq:sp-AIII-coefficient}, and \Cref{thm:sp-compact-vertical} proves the asymptotic.
	\end{proof}
	
	The local spectral constants of the one-column family are also explicit.
	\begin{theorem}
		For \eqref{eq:examples-hermitian-pair},
		\begin{equation}\label{zeta:hermitian-constants}
			\frac Q2=n(n+1),
			\qquad
			C_{G,K}=\left(\frac\pi2\right)^{n(n+1)},
			\qquad
			\beta_{G,K}=\frac{n(n+1)}4.
		\end{equation}
	\end{theorem}
	
	\begin{proof}
		Here $\dim K=n^2$ and $\dim\pLie=2n$. For type $A_{N-1}$ in the trace metric,
		\[
		|\rho_{\SU(N)}|^2=\frac{N(N^2-1)}{12},
		\]
		so subtraction of the $\SU(n)$ value gives $\beta_{G,K}=n(n+1)/4$.
		
		Under $A\mapsto\diag(A,-\Tr A)$, diagonalize $A=iS$ with eigenvalues $x_1,\ldots,x_n$, and put $u_i=x_i+\sum_j x_j$. Weyl integration reduces the relative Jacobian integral to
		\begin{equation*}
			\int_{\kLie}\mathcal D_{G,K}(Y)\,\mathrm dY
			=\frac{(2\pi)^{n(n-1)/2}}
			{\sqrt{n+1}\prod_{j=1}^n j!}
			\int_{\mathbb R^n}\prod_{i<j}(u_i-u_j)^2
			\prod_{i=1}^nw(u_i)\,\mathrm du.
		\end{equation*}
		The remaining matrix integral is
		\begin{equation*}
			\int_{\mathbb R^n}\prod_{i<j}(u_i-u_j)^2
			\prod_{i=1}^nw(u_i)\,\mathrm du
			=\pi^{n(n+1)}\left(\prod_{j=1}^n j!\right)^2.
		\end{equation*}
		This identity follows by applying Andr\'eief's formula \cite[Proposition~7.1]{Zygouras2022} to the Meixner-Pollaczek norms \cite[Equations~(18.19.7)-(18.19.9)]{DLMF}. For type $A_{N-1}$ the exponents are $1,\ldots,N-1$, and the trace metric Chevalley lattice has covolume $2^{N(N-1)/2}\sqrt N$. Macdonald's formula in the metric normalization of \cite{Hashimoto} therefore gives
		\begin{equation*}
			\vol(\SU(N))
			=\frac{\sqrt N(2\pi)^{(N-1)(N+2)/2}}
			{\prod_{j=1}^{N-1}j!}.
		\end{equation*}
		Substitution into \eqref{zeta:local-constants}, whose denominator is now $(4\pi)^{Q/2}$, gives the displayed value of $C_{G,K}$.
	\end{proof}
	
	\begin{example}
		We finish by recovering the CR sphere and rank-one cases. Let
		\[
		(G,K)=\bigl(\SU(n+1),S(\U(n)\times\U(1))\bigr)
		\]
		with the trace metric. Under the quotient identification
		\[
		L^2(G)^{\mathrm{right}\,\SU(n)}
		\simeq
		L^2\bigl(G/\SU(n)\bigr)
		\simeq
		L^2(\mathbb S^{2n+1}),
		\]
		let $\mathcal H^{p,q}$ denote the space of restrictions to $\mathbb S^{2n+1}$ of harmonic polynomials on $\mathbb C^{n+1}$ that are homogeneous of bidegree $(p,q)$. The right-$\SU(n)$-invariant $K$-types are one-dimensional, and the multiplicity-free $\U(n+1)\downarrow\U(n)$ branching rule identifies them with these spaces; see \cite{GoodmanWallach2009}. Hence
		\[
		L^2(\mathbb S^{2n+1})
		=
		\widehat{\bigoplus}_{p,q\ge0}\mathcal H^{p,q}.
		\]
		The Casimir difference gives the eigenvalue and the Weyl dimension formula gives the multiplicity:
		\[
		\Lambda_{p,q}=2pq+n(p+q),
		\qquad
		M_{p,q}
		=
		\frac{p+q+n}{n}
		\binom{p+n-1}{n-1}
		\binom{q+n-1}{n-1}.
		\]
		With the normalization used in \cite{BaudoinWang2013}, the corresponding positive eigenvalue is $2\Lambda_{p,q}$. Writing
		\[
		p=m+k_+,\qquad q=m+k_-,
		\qquad
		k_\pm:=\max\{\pm k,0\},
		\]
		one obtains
		\[
		2\Lambda_{p,q}
		=
		4m(m+|k|+n)+2n|k|,
		\]
		which is the CR sphere spectrum.
		
		For $n=1$, the right-$\SU(n)$ invariance condition is vacuous, so this is the full spectrum of the canonical contact sub-Laplacian on $\SU(2)$. In that case
		\[
		\Lambda_{p,q}=2pq+p+q,
		\qquad
		M_{p,q}=p+q+1,
		\]
		recovering, after the scalar metric normalization, the spectral parametrization used in \cite{BaudoinBonnefont2009,BauerFurutani2008}.
	\end{example}
	
	\section*{Acknowledgments}
	The first two authors thank L.~Hantzko and R.~Raussendorf for drawing their attention to the connection between sub-Riemannian geometry and measurement-based quantum computation.

	\noindent\textbf{Funding.}
	This work was supported by the National Natural Science Foundation of China (12301145, 12561020, 12261107) and the Yunnan Fundamental Research Projects (202401AU070123, 202601AT070048).

	\medskip
	{\bf Author Contributions:} All authors contributed equally to the writing and preparation of the manuscript.
	
	\medskip
	{\bf Data availability:} Data sharing is not applicable to this article as no new data were created or analyzed in this study.
	
	\medskip
	{\bf Conflict of Interests:} The authors declare that they have no conflict of interest.

\end{document}